\documentclass[11pt, twoside]{article}
\pdfoutput=1

\usepackage{graphicx}
\usepackage[caption=false]{subfig}
\usepackage{amsmath}
\usepackage{amssymb,amsfonts}
\usepackage{amsthm}
\usepackage{bm}
\usepackage{mathrsfs}
\usepackage{amssymb}
\usepackage{multirow}

\newcommand{\norm}[1]{\left\lVert#1\right\rVert}
\newcommand{\innOm}[2]{(#1,#2)_\Omega}
\newcommand{\innGa}[2]{\langle #1,#2\rangle_\Gamma}

\usepackage[margin=1in]{geometry}

\usepackage{algorithm}
\usepackage{algorithmic}

\numberwithin{equation}{section}
\usepackage{multirow}
\usepackage{makecell}
\usepackage{booktabs}
\theoremstyle{definition}
\newtheorem{theorem}{Theorem}

\newtheorem{lemma}{Lemma}
\newtheorem{corollary}{Corollary}
\newtheorem{definition}{Definition}
\newtheorem{remark}{Remark}

\usepackage{cite}
\usepackage{hyperref}
\usepackage[nameinlink]{cleveref}

\usepackage{fancyhdr}
\begin{document}
	
	\title{An HDG Method for Dirichlet Boundary  Control
		of Convection Dominated Diffusion PDE}

\author{Gang Chen%
	\thanks{School of Mathematics Sciences, University of Electronic Science and Technology of China, Chengdu, China (\mbox{cglwdm@uestc.edu.cn}).}
	\and
	John~R.~Singler%
	\thanks{Department of Mathematics and Statistics, Missouri University of Science and Technology, Rolla, MO, USA (\mbox{singlerj@mst.edu}).}
	\and
	Yangwen Zhang%
	\thanks{Department of Mathematics Science, University of Delaware, Newark, DE, USA (\mbox{ywzhangf@udel.edu}).}
}

\date{\today}

\maketitle

\begin{abstract}
		We first propose a hybridizable discontinuous Galerkin (HDG) method to approximate the solution of a \emph{convection dominated} Dirichlet boundary control problem. Dirichlet boundary control problems and  convection dominated problems are each very challenging numerically due to solutions with low regularity and sharp layers, respectively. Although there are some numerical analysis  works in the literature on \emph{diffusion dominated} convection diffusion Dirichlet boundary control problems, we are not aware of any existing numerical analysis works  for convection dominated boundary control problems. Moreover, the  existing numerical analysis techniques for convection dominated PDEs are not directly  applicable for the Dirichlet boundary control problem because of the low regularity solutions. In this work, we obtain an optimal a priori error estimate for the control under some conditions on the domain and the desired state. We also present some numerical experiments to illustrate  the performance of the HDG method for convection dominated  Dirichlet boundary control problems. 
\end{abstract}

\section{Introduction}
Let $\Omega\subset \mathbb{R}^{d}  \ (d= 2,3)$ be a Lipschitz polyhedral domain  with boundary $\Gamma = \partial \Omega$. We consider the following Dirichlet boundary control problem:
\begin{align}
\min J(y,u)=\frac{1}{2}\| y- y_{d}\|^2_{L^{2}(\Omega)}+\frac{\gamma}{2}\|u\|^2_{L^{2}(\Gamma)}, \quad \gamma>0, \label{cost1}
\end{align}
subject to
\begin{equation}\label{Ori_problem}
\begin{split}
-\varepsilon\Delta y+\nabla\cdot(\bm \beta y)+\sigma y &=f\quad\text{in}~\Omega,\\
y&=u\quad \text{on}~\Gamma,
\end{split}
\end{equation}
where $f\in L^2(\Omega), \varepsilon \ll 1$, and we make other assumptions on $\bm \beta $ and  $\sigma$ for our analysis.


Researchers have performed numerical analysis of computational methods for Dirichlet boundary control problems for over a decade. Many  researchers  considered the standard finite element method and obtained an error estimate for the optimal control of order $h^s$ for all $s<{\min\{1, {\pi}/{2\omega}\}}$, where $\omega$ is the largest angle of the boundary polygon; see, e.g., \cite{MR2272157,MR3070527,Mateos_Neitzel_Constrained_COA_2016}. Apel et al. in \cite{ApelMateosPfeffererRosch17} considered special meshes and obtained an optimal  convergence rate with  $s<{\min\{3/2, {\pi}/{\omega}-1/2\}}$.  Some mixed finite element methods have also been used for Dirichlet boundary control problems because the essential Dirichlet boundary condition becomes natural, i.e., the Dirichlet boundary data directly enters the variational setting.  In \cite{MR2806572}, Gong et al.\ used a standard mixed method to obtain  an  error estimate  for all $s<{\min\{1, {\pi}/{2\omega}\}}$. Recently, we used an HDG method to obtain an optimal convergence rate for all $s<{\min\{3/2, {\pi}/{\omega}-1/2\}}$ without using higher order elements or a special mesh \cite{HuShenSinglerZhangZheng_HDG_Dirichlet_control1}. Moreover, the number of degrees of freedom are lower for HDG methods than standard mixed methods.

All of the above works focus on Dirichlet boundary control of the Poisson equation. However, Dirichlet boundary control problems play an important role in many applications governed by more complicated models, such as the Navier-Stokes equations; see, e.g.,  \cite{MR1145711,MR1135991,MR1720145,MR1759904,MR1613873}. In order to work towards numerical analysis results for more difficult PDEs, one essential and necessary step is to fully understand the convection diffusion Dirichlet boundary control problem.  Benner and Y\"ucel in \cite{PeterBenner} used a local discontinuous Galerkin (LDG) method and  they obatined  an  error estimate for the control of order $\mathcal O(h^s)$ for all $s<{\min\{1, {\pi}/{2\omega}\}}$. Also, very recently, we proposed a new HDG method to  study  this problem and obtained an optimal convergence rate  $\mathcal O(h^s)$ for all $s<{\min\{3/2, {\pi}/{\omega}-1/2\}}$;  see \cite{HuMateosSinglerZhangZhang1,HuMateosSinglerZhangZhang2} for more details. 

However, the previous works only approximated  solutions of convection diffusion Dirichlet boundary control problems in the \emph{diffusion dominated} case. They did not consider the more difficult \emph{convection dominated} case, i.e., $\varepsilon \ll |\bm \beta| $.  Even without the Dirichlet boundary control,  solutions of convection dominated diffusion PDEs  typically  have layers; therefore, designing a robust numerical scheme for this problem is a major difficulty difficulty and has been considered in many works; see, e.g., \cite{MR0600325,MR890252,MR1031443,MR2068903} and the references therein. Discontinuous Galerkin (DG) methods have proved very useful for solving convection dominated PDEs; see, e.g.,  \cite{MR2970739,MR2585178,MR2257111,MR1655854,MR3614017,MR2945143,MR1885610} for standard  DG methods and  \cite{MR3487281,MR3342199} for HDG methods. For more information on HDG methods; see, e.g., \cite{MR2485455,MR2772094, MR3626531,MR3522968,MR3463051,MR3343926,MR3440284,QiuShi16NS}. Moreover, there are some existing convection dominated diffusion \emph{distributed} optimal control numerical analysis works; see, e.g., \cite{MR2302057,MR3022208,MR2595051}.  However, the techniques in the above works are \emph{not} applicable for convection dominated Dirichlet boundary control problems since the solutions of \eqref{cost1}-\eqref{Ori_problem} frequently have low regularity, i.e., $y\in H^{1+s}(\Omega)$ with $0\le s<1/2$.

Formally, the optimal control $ u \in L^2(\Gamma) $ and the optimal state $ y \in L^2(\Omega) $ minimizing the cost functional satisfy a mixed weak formulation of the  optimality system
\begin{subequations}\label{eq_adeq}
	\begin{align}
	-\varepsilon\Delta y+\nabla\cdot(\bm \beta y)+\sigma y&=f\qquad\quad~~\text{in}~\Omega,\label{eq_adeq_a}\\
	y&=u\qquad\quad~~\text{on}~\Gamma,\label{eq_adeq_b}\\
	-\varepsilon\Delta z-\nabla\cdot(\bm \beta z)+(\nabla\cdot\bm\beta+\sigma )z&=y-y_d\quad~~\text{in}~\Omega,\label{eq_adeq_c}\\
	z&=0\qquad\qquad\text{on}~\Gamma,\label{eq_adeq_d}\\
	\gamma u-\varepsilon\nabla z\cdot\bm n&=0\qquad\qquad\text{on}~\Gamma.\label{eq_adeq_e}
	\end{align}
\end{subequations}
In this  work, we use polynomials of degree $k$ to approximate the state $y$, dual state $z$ and their fluxes $\bm q = -\varepsilon\nabla  y $ and $ \bm p = -\varepsilon\nabla z$, respectively. Moreover, we also use polynomials of degree $k$ to approximate the numerical trace of the state and dual state on the edges (or faces) of the spatial mesh, which are the only globally coupled unknowns. The HDG  method considered here is different from the HDG  method we considered for convection diffusion Dirichlet boundary control problems in  \cite{HuMateosSinglerZhangZhang1,HuMateosSinglerZhangZhang2}. A major difference is that the HDG method here has a lower computational cost.

In \Cref{erroranalysis}, we obtain an optimal convergence rate for the optimal control in 2D under certain basic assumptions on the desired state $ y_d $ and the domain $ \Omega $; specifically, we prove
\begin{align}\label{error_u}
\| u - u_h\|_{\Gamma} \le C h^s,
\end{align}
for all $s<{\min\{3/2, {\pi}/{\omega}-1/2\}}$, and the constant $C$ only depends on the exact solution, the domain and the polynomial degree.  To prove the estimate \eqref{error_u}, we cannot use the numerical analysis strategy from \cite{PeterBenner,HuMateosSinglerZhangZhang1,HuMateosSinglerZhangZhang2} because the constants in their error estimates may blow up as $\varepsilon$ approaches zero. In order to obtain the  estimate \eqref{error_u} with the constant $C$ independent of  $\varepsilon$, we follow a strategy from \cite{MR3342199} and use  weighted test functions in an energy argument. However, the techniques used in \cite{MR3342199} are not directly applicable for solutions with low regularity. Moreover, unlike all the previous Dirichlet boundary control numerical analysis works, we only assume the mesh is shape regular, not quasi-uniform.  We present numerical results in \Cref{sec:numerics} to illustrate the performance of the HDG method.

\section{Optimality system, regularity and HDG formulation}
\label{sec:Background}

We begin with some notation. For any bounded domain $\Lambda \subset \mathbb{R}^d$ $(d=2,3)$, let $H^{m}(\Lambda)$ and $H^m_0(\Lambda)$  denote the usual  $m$th-order Sobolev spaces on $\Lambda$, and let $\|\cdot\|_{m, \Lambda}$, $|\cdot|_{m,\Lambda}$  denote the norm and  seminorm on these spaces.
We use $(\cdot,\cdot)_{m,\Lambda}$ to denote the inner product on $H^m(\Lambda)$, and set  $(\cdot,\cdot)_{\Lambda}:=(\cdot,\cdot)_{0,\Lambda}$.
When $\Lambda=\Omega$, we denote $\|\cdot\|_{m }:=\|\cdot\|_{m, \Omega}$, and $ |\cdot|_{m}:=|\cdot|_{m,\Omega}$.  Also, when $\Lambda$ is the boundary of a set in $\mathbb R^d$,  we use $\langle\cdot,\cdot\rangle_{\Lambda}$ to replace $(\cdot,\cdot)_{\Lambda}$. Bold face fonts will be used for vector Sobolev spaces along with vector-valued functions. In addition, we introduce the following  space:
\begin{align*}
\bm{H}(\text{div},\Lambda):=\{\bm{v}\in [L^2(\Lambda)]^d: \nabla\cdot\bm{v}\in L^{2}(\Lambda)\}.
\end{align*}

We now present the optimality system for  problem \eqref{cost1}-\eqref{Ori_problem} and give a regularity result.
\subsection{Optimality system and regularity}
\label{sec:Analysis_of_the_Dirichlet_Control_Problem} 
Throughout the paper,  we suppose $\Omega$ is a convex polygonal domain, and let $\omega \in [\pi/3, \pi)$ denote its largest interior angle. The optimal control $ u $ is determined by the optimality system for the state $ y $ and the dual state $ z $.  For the HDG method, we use a mixed formulation of the optimality system; therefore we introduce the primary flux $ \bm q = -\varepsilon \nabla y $ and the dual flux $ \bm p = -\varepsilon \nabla z $.  The well-posedness and regularity of the mixed formulation of the optimality system is contained in the result below. The proof of \Cref{MT210}  is omitted here since it is very similar with a proof of a similar result in  \cite{HuMateosSinglerZhangZhang1}.
\begin{theorem}\label{MT210}
	If $y_d\in H^{t^*}(\Omega)$ for some $0\leq t^*<1$, $\sigma\in L^{\infty}(\Omega)\cap H^1(\Omega)$, $f=0$ and the velocity vector field $\bm \beta$ satisfies
	\begin{equation}\label{eqn:beta_assumptions1}
	\bm \beta \in [L^\infty(\Omega)]^2,  \quad  \nabla\cdot\bm{\beta}\in L^\infty(\Omega),  \quad \sigma + \frac 1 2 \nabla \cdot \bm \beta \geq 0,  \quad  \nabla \nabla \cdot \bm{\beta}\in[L^2(\Omega)]^2,
	\end{equation}
	then problem \eqref{cost1}-\eqref{Ori_problem} has a unique solution $  u\in L^2(\Gamma)$.  Moreover, for any $ s > 0 $ satisfying $s\leq 1/2 +t^*$ and $s<\min\{ 3/ 2,{\pi}/{\omega}- 1/ 2\}$, we have 
	\begin{align*}
	(u,\bm q, \bm p, y, z) &\in H^s (\Gamma)\times\bm H^{s-\frac{1}{2}} (\Omega)\times \bm H^{s+ \frac 1 2}(\Omega) \times H^{s+\frac 1 2}(\Omega)\times (H^{s+ \frac 3 2}(\Omega)\cap H_0^1(\Omega)),
	\end{align*}
	is the unique solution of 
	\begin{subequations}\label{mixed}
		\begin{align}
		\varepsilon^{-1}\innOm{ {\bm q}}{\bm r} -\innOm{ y}{\nabla\cdot \bm r} + \innGa{ u}{\bm r\cdot \bm n} &= 0,\label{mixed_a1}\\
		\innOm{\nabla\cdot( {\bm q}+\bm{\beta}  y)}{w} +\innOm{\sigma y}{w}  &= 0,\label{mixed_b1}\\
		\varepsilon^{-1}\innOm{ {\bm p}}{\bm r}-\innOm{ z}{\nabla\cdot \bm r} &= 0,\label{mixed_c1}\\
		\innOm{\nabla\cdot( {\bm p}-\bm{\beta}  z)}{w} +\innOm{(\nabla\cdot\bm \beta+\sigma) y}{w} &= \innOm{ y- {y_d}}{w},\label{mixed_d1}\\
		\innGa{\gamma  u+{\bm p}\cdot\bm n}{v} &=  0,\label{mixed_e1}
		\end{align}
		for all $(\bm r,w,v)\in \bm H(\textup{div},\Omega)\times L^2(\Omega)\times L^2(\Gamma)$.
		Furthermore, we have $\Delta y \in L^2(\Omega)$.
	\end{subequations}
\end{theorem}


\subsection{The HDG formulation} 
Let $\mathcal{T}_h=\bigcup\{T\}$  be a conforming simplex mesh that partitions the domain $\Omega$. For any $T\in\mathcal{T}_h$, we let $h_T$ be the diameter of  $T$ and denote the mesh size by $h:=\max_{T\in\mathcal{T}_h}h_T$.  
Denote the edges of $T$ by $E$,  let $\mathcal E_h$ be the set of all edges $E$, let  $\mathcal E_h^\partial$ be the set of edges $E$ such that $E\subset\Gamma$, and set $\mathcal E_h^o = \mathcal E_h \setminus \mathcal E_h^\partial$.  
Let $h_E$ denote the diameter of $E$.
The mesh dependent inner products are denoted by
\begin{align*}
(w,v)_{\mathcal{T}_h} = \sum_{T\in\mathcal{T}_h} (w,v)_T,   \quad\quad\quad\quad\left\langle \zeta,\rho\right\rangle_{\partial\mathcal{T}_h} = \sum_{T\in\mathcal{T}_h} \left\langle \zeta,\rho\right\rangle_{\partial T}.
\end{align*}
We use $\nabla $ and $\nabla \cdot$ to denote the broken gradient and broken divergence with respect to
$\mathcal{T}_h$.
For an integer $k\ge 0$, $\mathcal {P}_k(\Lambda)$  denotes the set of all  polynomials  defined on $\Lambda$ with degree not greater than $k$. We introduce the discontinuous finite element spaces.
\begin{align*}
\bm{V}_h  &:= \{\bm{v}\in [L^2(\Omega)]^2: \bm{v}|_{T}\in [\mathcal {P}_k(T)]^d, \forall T\in \mathcal{T}_h\},\\
{W}_h  &:= \{{w}\in L^2(\Omega): {w}|_{T}\in \mathcal {P}_k(T), \forall T\in \mathcal{T}_h\},\\
M^o_h  &:= \{{\mu}\in L^2(\mathcal{E}_h): {\mu}|_{E}\in \mathcal {P}_k(E), \forall E\in \mathcal{E}_h, \text{ and }\mu|_{\Gamma}=0\},\\
M_h^\partial  &:= \{{\mu}\in L^2(\mathcal{E}^{\partial}_h): {\mu}|_{E}\in \mathcal {P}_k(E), \forall E\in \mathcal{E}_h^{\partial}\}.
\end{align*}
In our earlier works \cite{HuMateosSinglerZhangZhang1,HuMateosSinglerZhangZhang2}, we used a  $\mathcal P_{k+1}$ local space for the spaces $W_h$ and $M_h$. In this work, we use polynomial degree $k$ for all spaces. Since the globally coupled degrees of freedom depend on the space $M_h$, the computational cost of the HDG method in this paper is much lower than the HDG method  in \cite{HuMateosSinglerZhangZhang1,HuMateosSinglerZhangZhang2}.

The HDG method for mixed weak form of the  optimality system \eqref{mixed} is to find   $(\bm q_h,y_h,\widehat{y}^o_h,\bm p_h,z_h,\widehat{z}^o_h,u_h)$ $\in [\bm V_h\times W_h\times M^{o}_h]^2\times M^{\partial}_h$ such that
\begin{subequations}\label{HDG_discrete2}
	\begin{align}
	\varepsilon^{-1}(\bm q_h,\bm r_h)_{\mathcal{T}_h} -( y_h,\nabla \cdot\bm r_h)_{\mathcal{T}_h} +\langle \widehat y^o_h,\bm r_h\cdot\bm n\rangle_{\partial\mathcal T_h}=-\langle u_h,\bm r_h\cdot\bm{n} \rangle_{\mathcal{E}_h^\partial},\label{HDG_discrete2_a}
	\end{align}
	for all $\bm r_h \in \bm V_h$,
	\begin{align}\label{HDG_discrete2_b}
	\begin{split}
	&\quad-(w_h,\nabla \cdot\bm{q}_h)_{\mathcal{T}_h}+\langle \widehat{w}_h^o,\bm{q}_h\cdot\bm{n} \rangle_{\partial\mathcal{T}_h}
	-\langle \tau_1(  y_h-\widehat{y}_h^o), w_h-\widehat{w}_h^o \rangle_{\partial\mathcal{T}_h}\\
	&\quad+( y_h,\bm\beta\cdot\nabla  w_h)_{\mathcal{T}_h}-\langle \widehat{y}^o_h,\bm \beta\cdot\bm n w_h\rangle_{\partial\mathcal{T}_h}
	-(\sigma y_h,w_h)_{\mathcal{T}_h}\\
	&=-(f,w_h)_{\mathcal{T}_h}-\langle  (\tau_1- \bm \beta \cdot\bm n) u_h, v_h\rangle_{\mathcal{E}_h^{\partial}},  
	\end{split}
	\end{align}
	for all $(w_h,\widehat w_h^o)\in W_h\times M^o_h$,
	\begin{align}
	\varepsilon^{-1}(\bm p_h,\bm r_h)_{\mathcal{T}_h} -( z_h,\nabla \cdot\bm r_h)_{\mathcal{T}_h} +\langle \widehat z^o_h,\bm r_h\cdot\bm n\rangle_{\partial\mathcal T_h}&=0,\label{HDG_discrete2_c}
	\end{align}
	for all $r_h\in \bm V_h$,
	\begin{align} \label{HDG_discrete2_d}
	\begin{split}
	&\quad -(w_h,\nabla \cdot\bm{p}_h)_{\mathcal{T}_h}+\langle \widehat{w}_h^o,\bm{p}_h\cdot\bm{n} \rangle_{\partial\mathcal{T}_h} -\langle \tau_2(  z_h-\widehat{z}_h^o)_{\mathcal{T}_h}, w_h-\widehat{w}_h^o \rangle_{\partial\mathcal{T}_h}\\
	&\quad-( z_h,\bm\beta\cdot\nabla  w_h)_{\mathcal{T}_h}+\langle \widehat{z}_h^o,\bm \beta\cdot\bm n  w_h\rangle_{\partial\mathcal{T}_h}
	-((\sigma+\nabla\cdot\bm{\beta}) z_h,w_h)_{\mathcal{T}_h}\\
	&=-(y_h-y_d,{w}_h)_{\mathcal{T}_h}, 
	\end{split}
	\end{align}
	for all  $(w_h,\widehat w_h^o)\in W_h\times M^o_h$,
	\begin{align}
	\langle \gamma u_h+\bm p_h\cdot\bm n+\tau_2( z_h-\widehat{z}_h^o), \widehat w_h^{\partial}\rangle_{\varepsilon_h^\partial}& =0,\label{6e}
	\end{align}
	for all $ \widehat w_h^{\partial} \in  M_h^{\partial}$. Here, the positive stabilization functions $\tau_1$ and $\tau_2$ are chosen as 
	\begin{align}
	\tau_1 |_{\partial T}&= \|\bm \beta\cdot\bm n\|_{0,\infty,\partial T}+\frac{1}{2}\bm\beta\cdot\bm{n}+\varepsilon h_T^{-1},\label{def_tau1}\\
	\tau_2 |_{\partial T}&= \|\bm \beta\cdot\bm n\|_{0,\infty,\partial T}-\frac{1}{2}\bm\beta\cdot\bm{n}+\varepsilon h_T^{-1}.\label{def_tau2}
	\end{align}
	To simplify the presentation later, we define 
	\begin{align}\label{def_tau}
	\tau = \frac{\tau_1 + \tau_2}{2} = \|\bm \beta\cdot\bm n\|_{0,\infty,\partial T}+\varepsilon h_T^{-1}.
	\end{align}
\end{subequations}

\subsection{A compact formulation}
To simplify the notation, for $(\bm q_h, y_h, \widehat y_h^o, \\ \bm p_h, z_h,\widehat z_h^o, \bm r_h, w_h,\widehat w_h^o) \in [\bm V_h\times W_h\times M_h^o]^3$, we denote
\begin{align}
\hspace{1em}&\hspace{-1em}\mathcal B_1(\bm{q}_h,y_h,\widehat y_h^o;\bm r_h,w_h,\widehat w_h^o)\nonumber\\
&=\varepsilon^{-1}(\bm q_h,\bm r_h)_{\mathcal{T}_h} -( y_h,\nabla \cdot\bm r_h)_{\mathcal{T}_h} +\langle \widehat y^o_h,\bm r_h\cdot\bm n\rangle_{\partial\mathcal T_h} \nonumber\\
&\quad-(w_h,\nabla \cdot\bm{q}_h)_{\mathcal{T}_h}+\langle \widehat{w}_h^o,\bm{q}_h\cdot\bm{n} \rangle_{\partial\mathcal{T}_h} -\langle \tau_1(  y_h-\widehat{y}_h^o), w_h-\widehat{w}_h^o \rangle_{\partial\mathcal{T}_h}\nonumber\\
&\quad +( y_h,\bm\beta\cdot\nabla w_h)_{\mathcal{T}_h}-\langle \widehat{y}^o_h,\bm \beta\cdot\bm n w_h\rangle_{\partial\mathcal{T}_h}
-(\sigma y_h,w_h)_{\mathcal{T}_h},\label{def_B1}\\
\hspace{1em}&\hspace{-1em}\mathcal B_2(\bm{p}_h,z_h,\widehat z_h^o;\bm{r}_h,w_h,\widehat w_h^o)\nonumber\\
&=\varepsilon^{-1}(\bm p_h,\bm r_h)_{\mathcal{T}_h} -( z_h,\nabla \cdot\bm r_h)_{\mathcal{T}_h} +\langle \widehat z^o_h,\bm r_h\cdot\bm n\rangle_{\partial\mathcal T_h}\nonumber\\
&\quad-(w_h,\nabla \cdot\bm{p}_h)_{\mathcal{T}_h}+\langle \widehat{w}_h^o,\bm{p}_h\cdot\bm{n} \rangle_{\partial\mathcal{T}_h} -\langle \tau_2(  z_h-\widehat{z}_h^o), w_h-\widehat{w}_h^o \rangle_{\partial\mathcal{T}_h}\nonumber\\
&\quad -( z_h,\bm\beta\cdot\nabla  w_h)_{\mathcal{T}_h}+\langle \widehat{z}_h^o,\bm \beta\cdot\bm n  w_h\rangle_{\partial\mathcal{T}_h}
-((\sigma+\nabla\cdot\bm{\beta}) z_h,w_h)_{\mathcal{T}_h}. \label{def_B2}
\end{align}

Then we can rewrite \eqref{HDG_discrete2} as follows:
find $\left(\bm q_h,y_h,\widehat{y}_h^o,\bm p_h,z_h,\widehat{z}_h^o,u_h\right)\in [\bm V_h\times W_h\times M^{o}_h]^2\times M^{\partial}_h$ such that
\begin{subequations}\label{HDG_full_discrete}
	\begin{align}
	\mathcal B_1(\bm q_h,y_h,\widehat y^o_h;\bm r_1,w_1,\widehat w_1^o)&=-( f, w_1)_{\mathcal T_h}-\langle  u_h, \tau_2 w_1+\bm r_1\cdot\bm{n}\rangle_{\mathcal{E}_h^{\partial}},\label{HDG_full_discrete_a}\\
	\mathcal B_2(\bm p_h,z_h,\widehat z^o_h;\bm r_2,w_2,\widehat w_2^o)&=-(y_h-y_d,w_2)_{\mathcal{T}_h},\label{HDG_full_discrete_b}\\
	\langle \bm p_h\cdot\bm n+\tau_2( z_h-\widehat{z}_h^o), \widehat w_h^\partial\rangle_{\varepsilon_h^\partial} &=\langle \gamma u_h, \widehat w_h^\partial\rangle_{\varepsilon_h^\partial} ,\label{HDG_full_discrete_c}
	\end{align}
	for all  $\left(\bm r_1,w_1,\widehat{w}^o_1,\bm r_2,w_2,\widehat{w}^o_2,  \widehat w_h^\partial \right)\in [\bm V_h\times W_h\times M^{o}_h]^2\times M^{\partial}_h$.
\end{subequations}

The following basic result, which is similar to results in \cite{HuMateosSinglerZhangZhang1,HuMateosSinglerZhangZhang2},  is crucial to the proof of the well-posedness of the discrete optimality system \eqref{HDG_discrete2_a}-\eqref{6e}, and is also a very important part of  the final stage of numerical analysis (see the proof of \Cref{lemma_u_par}).
\begin{lemma}\label{sys}
	For all $\left(\bm q_h,y_h,\widehat{y}_h^o,\bm r_h,w_h,\widehat{w}_h^o\right)\in [\bm V_h\times W_h\times M^{o}_h]^2$, we have 
	\begin{eqnarray}
	\mathcal B_1(\bm{q}_h,y_h,\widehat y_h^o;\bm r_h,w_h,\widehat w_h^o)=
	\mathcal B_2(\bm r_h,w_h,\widehat w_h^o;\bm{q}_h,y_h,\widehat y_h^o).\label{sym}
	\end{eqnarray}
\end{lemma}
\begin{proof} Using the definitions in \eqref{def_B1}-\eqref{def_B2} and integration by parts give
	\begin{eqnarray}
	&&\mathcal B_1(\bm{q}_h,y_h,\widehat y_h^o;\bm r_h,w_h,\widehat w_h^o)-
	\mathcal B_2(\bm r_h,w_h,\widehat w_h^o;\bm{q}_h,y_h,\widehat y_h^o)\nonumber\\
	&&\qquad=-\langle \tau_1(  y_h-\widehat{y}_h^o), w_h-\widehat{w}_h^o \rangle_{\partial\mathcal{T}_h}\nonumber\\
	&&\qquad\quad+\langle \tau_2(y_h-\widehat{y}_h^o), w_h-\widehat{w}_h^o \rangle_{\partial\mathcal{T}_h} \nonumber\\
	&&\qquad\quad+( y_h,\bm\beta\cdot\nabla  w_h)_{\mathcal{T}_h}-\langle \widehat{y}^o_h,\bm \beta\cdot\bm n w_h\rangle_{\partial\mathcal{T}_h}
	-(\sigma y_h,w_h)_{\mathcal{T}_h}\nonumber\\
	&&\qquad\quad+( w_h,\bm\beta\cdot\nabla  y_h)_{\mathcal{T}_h}-\langle \widehat{w}_h^o,\bm \beta\cdot\bm n  y_h\rangle_{\partial\mathcal{T}_h}
	+((\sigma+\nabla\cdot\bm{\beta}) w_h,y_h)_{\mathcal{T}_h} \nonumber\\
	&&\qquad=-\langle \bm\beta\cdot\bm n( y_h-\widehat{y}_h^o), w_h-\widehat{w}_h^o \rangle_{\partial\mathcal{T}_h}\nonumber\\
	&&\qquad\quad+\langle y_h,\bm \beta\cdot\bm n w_h\rangle_{\partial\mathcal{T}_h}
	-\langle \widehat y^o_h,\bm \beta\cdot\bm n w_h\rangle_{\partial\mathcal{T}_h}
	-\langle \widehat w^o_h,\bm \beta\cdot\bm n y_h\rangle_{\partial\mathcal{T}_h}\nonumber\\
	&&\qquad=0,\nonumber
	\end{eqnarray}
	where we used $\langle 
	\bm\beta\cdot\bm n,\widehat{y}_h^o\widehat{w}_h^o
	\rangle_{\partial\mathcal{T}_h}=0$.
	This proves our result.
\end{proof}

\section{Stability}
To perform the stability and error analysis for the convection dominated boundary control problem, we need to assume some conditions on the velocity vector field $\bm{\beta}$ and the effective reaction function $\bar{\sigma}:=\sigma+\frac{1}{2}\nabla\cdot\bm \beta$.

%

\begin{description}	
	\item[\textbf{(A1)}] $\bar{\sigma}$ has a nonnegative lower bound, i.e,
	\begin{align}\label{sigma}
	\sigma_0:=\inf_{\bm{x}\in\Omega}\bar{\sigma}\ge 0.
	\end{align}
	
	\item[\textbf{(A2)}] $\bm \beta$ has no closed curves and
	\begin{align*}
	|\bm \beta(\bm{x})| \neq  0 \text{ for all } \bm x\in\Omega.
	\end{align*}
	
	\item[\textbf{(A3)}]  $\varepsilon<\min_{T\in \mathcal T_h}\{h_T\}.$
\end{description}
We note that we have already assumed \textbf{(A1)}  in \Cref{eqn:beta_assumptions1} in \Cref{MT210}.  We repeat the assumption here to highlight it. Also, since we are interested in the convection dominated case, \textbf{(A3)}  is a reasonable assumption. As shown in \cite{MR2485457}, assumption \textbf{(A2)} implies for any integer $k \geq 0$,  there exists a function $\psi\in W^{k+1,\infty}(\Omega)$ such that  for all $ \bm{x}\in\Omega$, we have 
\begin{align}\label{condition-beta}
\bm \beta\cdot\nabla\psi\ge2 \beta_0>0,
\end{align}
where  $\beta_0:={\|\bm \beta\|_{0,\infty}}/{L}$ and   $L$  is the diameter of $\Omega$. We  use assumption \textbf{(A3)} in the analysis to remove the assumption on the meshes. Specifically, in the proofs of \Cref{energyforE1andE2} and \Cref{error_ee}, we use assumption \textbf{(A3)} and a local inverse inequality to replace a global inverse inequality that has been used in all previous Dirichlet boundary control works. Therefore,  we only assume $ \{ \mathcal T_h \} $ is a conforming simplex partition of $ \Omega $. All previous works on Dirichlet boundary control problems required a conforming quasi-uniform mesh. In the future,  we hope to performed an {\em {a  posteriori}}  error analysis for the convection dominated boundary control problem.

\begin{remark}	
	If $\sigma_0\geq 0$, then assumption \textbf{(A2)} is the minimal known requirement that can be used to establish stability and error analysis results for numerical methods; see, e.g., \cite{MR2485457,MR3342199}. If instead $\sigma_0> 0$,  then we don't need to assume \textbf{(A2)} and the numerical analysis is less technical.  Specifically, we don't need to prove \Cref{LBB} below if $\sigma_0\geq 0$.
\end{remark}



%


\subsection{Preliminary material}
For any nonnegative integer $j$, we define the $L^2$-projections $\Pi_j^o$ and $\Pi_j^{\partial}$ as follows: for any $T\in \mathcal{T}_h$, $E\subset\partial T$,  $v\in L^2(T)$, $q\in L^2(E)$, find  $\Pi_j^ov\in \mathcal  {P}_{j}(T)$ and $\Pi_j^{\partial}q\in\mathcal  P_j(E)$ satisfying
\begin{subequations}\label{orth}
	\begin{align}
	(\Pi_j^ov,w_j)_T&=(v, w_j)_T,\ \ \forall w_j\in \mathcal P_j(T),\label{orth_o}\\
	\langle\Pi_j^{\partial}q,r_j\rangle_E&=\langle q, r_j\rangle_E,\ \ \ \forall r_j\in \mathcal P_j(E).\label{orth_p}
	\end{align}
\end{subequations}
We also define $\widetilde{\Pi}^{\partial}_k$ as
\begin{eqnarray}
\widetilde{\Pi}^{\partial}_k|_E=\left\{
\begin{aligned}
{\Pi}^{\partial}_k|_E, \ \  E\in\mathcal{E}_h^o,\\
0,                         \qquad E\in\mathcal{E}_h^{\partial}.
\end{aligned}
\right.\nonumber
\end{eqnarray}
Then $\widetilde{\Pi}^{\partial}_k$ is an operator mapping $L^2(\mathcal{E}_h)$ to $M_h^o$. 





We first give an approximation property from \cite[Theorem 4.3.8, Proposition 4.1.9]{Brenner_FEMBook_2008}, and then we prove the basic stability and approximation properties for $L^2$ projections.

\begin{lemma} 
	Let $m\geq 1$ be an integer. For any  $T\in\mathcal{T}_h$, $v\in H^m(T)$ and integer $s$ satisfying $0\le s\le m$, there exists $I_{m-1}v\in \mathcal{P}_{m-1}(T)$ such that
	\begin{subequations}
		\begin{align}
		|v-I_{m-1}v|_{s,T}&\le Ch_T^{m-s}|v|_{m,T},\label{I}\\
		\|v-I_{m-1}v\|_{0,\infty,T}&\le Ch_T|v|_{1,\infty,T}.\label{I2}
		\end{align}
	\end{subequations}
\end{lemma}

\begin{lemma} \label{approximation} 
	Let $s$ be a real number. For any nonnegative integer $j$, let $m$ be a real number satisfying $\frac{1}{2}<m\le j+1$ and let $\ell \in  \{0,1\}$. For all $T\in\mathcal{T}_h$, $E\in\mathcal{E}_h$, it holds
	\begin{subequations}
		\begin{align}
		|\Pi^{o}_{\ell}v|_{j, T}&\le C|v|_{j,T}, &\forall v\in H^j(T),\label{stab_1}\\
		\|\Pi^{\partial}_jv\|_{E}&\le \|v\|_{E},&\forall v\in L^{2}(E),\label{stab_2}\\
		|v-\Pi^{o}_jv|_{s,T}&\le C h_T^{m-s}|v|_{m,T},&\forall v\in H^{m}(T),\ 0\le s \le m,\label{stab_4}\\
		|v-\Pi_0^ov|_{\ell,\infty,T}&\le 
		Ch_T^{1-\ell}|v|_{1,\infty,T},& \forall v\in W^{1,\infty}(T), \label{stab_5}\\
		|v-\Pi^{o}_jv|_{s,\partial T}&\le C h_T^{m-s-1/2}|v|_{m,T},&\forall v\in H^{m}(T), 0\le s+1 \le m,\label{stab_6}\\
		\|w\|_{\partial T}&\le C h_T^{-1/2}\|w\|_{T},&\forall w\in W_h.\label{stab_7}
		\end{align}
	\end{subequations}
\end{lemma}
\begin{proof} 
	\Cref{stab_1} follows from \Cref{stab_4};  \Cref{stab_2}  follows from the definition of $L^2$ projection;  \Cref{stab_6} follows from \Cref{stab_4} and the  trace inequality;  and \Cref{stab_7} follows from the   trace inequality and inverse inequality. The only thing left is to prove \Cref{stab_4} and \Cref{stab_5}. 
	
	For  \Cref{stab_4}, in view of \Cref{I}, an inverse inequality, and the fact that $\|\Pi_{j}^ov\|_{0,T}\le \|v\|_{0,T}$, for $1\le m\le j+1$ we have
	\begin{align*}
	|v-\Pi_j^ov|_{s,T}&\le |v-I_{m-1}v|_{s,T}+|I_{m-1}v-\Pi_j^ov|_{s,T} \nonumber\\
	&=	  |v-I_{m-1}v|_{s,T}+|\Pi_j^o(I_{m-1}v-v)|_{s,T} \nonumber\\
	&\le   |v-I_{m-1}v|_{s,T}+Ch_T^{-s}\|\Pi_j^o(I_{m-1}v-v)\|_{0,T} \nonumber\\
	&\le   |v-I_{m-1}v|_{s,T}+Ch_T^{-s}\|I_{m-1}v-v\|_{0,T} \nonumber\\
	&\le Ch_T^{m-s}\|v\|_{m,T}.
	\end{align*}
	As for \Cref{stab_5}, $\ell=1$ is obvious and therefore we set $\ell=0$. By a standard scaling argument, the following stability result holds:
	\begin{align} \label{infty}
	\|\Pi_0^ov\|_{0,\infty,T}\le C 	\|v\|_{0,\infty,T}.
	\end{align}
	By an inverse inequality, \eqref{infty}, and \eqref{I2} we get 
	\begin{align*}
	\|v-\Pi_0^ov\|_{0,\infty,T}&\le \|v-I_{0}v\|_{0,\infty,T}+\|I_{0}v-\Pi_0^ov\|_{0,\infty,T} \nonumber\\
	&=\|v-I_{0}v\|_{0,\infty,T}+\|\Pi_0^o(I_{0}v-v)\|_{0,\infty,T} \nonumber\\
	&\le C\|v-I_{0}v\|_{0,\infty,T}\nonumber\\
	&\le Ch_T|v|_{1,\infty,T}.
	\end{align*}
\end{proof}

In addition, we have the following  super-approximation results; a similar result can be found in \cite{MR2485457}. 
\begin{lemma}\label{Super-approximation}
	Let $T\in\mathcal{T}_h,E\subset\partial T$. Then, for any $u_h\in \mathcal  P_{k}(T)$ and $\eta\in W^{1,\infty}(T)$, there holds:
	\begin{subequations}
		\begin{align}
		\|\eta u_h-\Pi^{o}_k(\eta u_h)\|_{T}&\le C h_T|\eta|_{1,\infty,T}\|u_h\|_{T},\label{compo_0T}\\
		|\eta u_h-\Pi^{o}_k(\eta u_h)|_{1,T}&\le C |\eta|_{1,\infty,T}\|u_h\|_{T}, \label{compo_1T}\\
		\|\eta u_h-\Pi^{o}_k(\eta u_h)\|_{\partial T}&\le C h_T^{1/2}|\eta|_{1,\infty,T}\|u_h\|_{T},\label{compo_0TE}\\
		\|\eta u_h-\Pi^{\partial}_k(\eta u_h)\|_{E}&\le C h_T|\eta|_{1,\infty,T}\|u_h\|_{E}.\label{compo_0EE}
		\end{align}
	\end{subequations}
\end{lemma}
\begin{proof}
	We notice that $\eqref{compo_0TE}$  follows from $\eqref{compo_0T}$, $\eqref{compo_1T}$, and the trace inequality. Next, for  $j\in \{0,1\}$, we have 
	\begin{align}
	|\eta u_h-\Pi^{o}_k(\eta u_h)|_{j,T}&\le |\eta u_h-(\Pi^{o}_0\eta) u_h|_{j,T}+|(\Pi^{o}_0\eta) u_h-\Pi^{o}_k(\eta u_h)|_{j,T} \nonumber  \\
	&=|(\eta-\Pi^{o}_0\eta) u_h|_{j,T}+|\Pi^{o}_k((\Pi^{o}_0\eta) u_h-\eta u_h)|_{j,T}  \nonumber \\
	&\le C |(\eta-\Pi^{o}_0\eta) u_h|_{j,T} \nonumber & \textup{by} \  \eqref{stab_1},\\
	&\le C |\eta-\Pi^{o}_0\eta|_{j,\infty}\|u_h\|_{T}+\|\eta-\Pi^{o}_0\eta\|_{0,\infty}|u_h|_{j,T}
	\nonumber\\
	&\le Ch_T^{1-j}|\eta|_{1,\infty,T}\|u_h\|_{T},\nonumber    & \textup{by} \  \eqref{stab_5}.
	\end{align}
	This proves \eqref{compo_0T} and \eqref{compo_1T}.  Similarly, for $E\subset \partial T$, we have 
	\begin{align}
	\|\eta u_h-\Pi^{\partial}_k(\eta u_h)\|_{E}&\le \|\eta u_h-(\Pi^{o}_0\eta) u_h\|_{E}+\|(\Pi^{o}_0\eta) u_h-\Pi^{\partial}_k(\eta u_h)\|_{E} \nonumber\\
	&=\|(\eta-\Pi^{o}_0\eta) u_h\|_{E}+\|\Pi^{\partial}_k((\Pi^{o}_0\eta) u_h-\eta u_h)\|_{E}  \nonumber\\
	&\le C\|(\eta-\Pi^{o}_0\eta) u_h\|_{E} \nonumber& \textup{by} \  \eqref{stab_2},\\
	&\le C \|\eta-\Pi^{o}_0\eta\|_{0,\infty,T}\| u_h\|_{E} \nonumber\\
	&\le C h_T|\eta|_{1,\infty,T}\|u_h\|_{E}\nonumber  & \textup{by} \  \eqref{stab_5}.
	\end{align}
	This proves \eqref{compo_0EE}.
\end{proof}

For the analysis of the low regularity case, we need the following result from \cite{MR3508837}:
\begin{lemma}
	If $\ell\ge 1$ is an integer that is large enough, then there exists an  interpolation operator  $\mathcal{I}^c_h: W_h\times M^o_h\to H^1_0(\Omega)\cap \mathcal {P}^{\mathcal{T}_h}_{k+\ell}$ such that for all $(w_h,\widehat{w}^o_h,v_h,\widehat{v}_h^o)\in [W_h\times M_h^o]^2$,
	for all  $T\in\mathcal T_h$ and  for all $E\in \mathcal E_h$, we have 
	\begin{subequations}
		\begin{align}
		(\mathcal{I}_h^c(w_h,\widehat w^o_h), v_h)_T&=(w_h,v_h)_T, \label{I*1}\\
		\langle\mathcal{I}_h^c(w_h,\widehat{w}^o_h),\widehat v_h^o\rangle_{E}&=\langle \widehat{w}^o_h,\widehat{v}^o_h\rangle_{E},\label{I*2}\\
		\|\nabla\mathcal{I}_h^c(w_h,\widehat{w}^o_h)\|_{\mathcal{T}_h}&\le  C\left(\|\nabla  w_h\|_{\mathcal{T}_h}+\|h_E^{-1/2}( w_h-\widehat{w}^o_h )\|_{\partial \mathcal{T}_h}\right)\label{I*3},\\
		\|w_h-\mathcal{I}_h^c(w_h,\widehat{w}^o_h)\|_{\mathcal{T}_h}&\le  Ch\left(\|\nabla  w_h\|_{\mathcal{T}_h}+\|h_E^{-1/2}( w_h-\widehat{w}^o_h )\|_{\partial \mathcal{T}_h}\right),\label{I*4}
		\end{align}
		where $\mathcal P^{\mathcal{T}_h}_{k+\ell}=\{w_h\in L^2(\Omega): w_h|_T\in\mathcal P_{k+\ell}(T),\ \forall T\in\mathcal{T}_h\}$.
	\end{subequations}
\end{lemma}

\subsection{Proof of the stability of \eqref{HDG_full_discrete}}
Next, we present the stability of the above HDG method for the convection dominated Dirichlet boundary control problem. We follow a similar strategy  to \cite{MR2485457,MR3342199}. We first collect some basic equalities and inequalities, which are used frequently in our paper.

\begin{lemma}
	For all $(w_h,\widehat w_h^o)\in W_h\times M_h^o$, we have 
	\begin{subequations}
		\begin{gather}
		(w_h, \bm \beta \cdot \nabla w_h)_{\mathcal T_h} = \frac 1 2 \langle  \bm{\beta}\cdot \bm n w_h, w_h \rangle_{\partial\mathcal T_h} - \frac 1 2 (\nabla \cdot\bm \beta w_h, w_h)_{\mathcal T_h},\label{beta1}\\
		\frac 1 2 \langle \bm\beta \cdot\bm{n}  w_h, w_h \rangle_{\partial \mathcal{T}_h} -\langle \widehat{w}^o_h,\bm \beta\cdot\bm n w_h\rangle_{\partial\mathcal{T}_h} = \frac 1 2 \langle \bm\beta \cdot\bm{n}  (w_h - \widehat w_h^o), w_h - \widehat w_h^o \rangle_{\partial \mathcal{T}_h}, \label{beta2}\\
		\|w_{h}\|^2_{\mathcal{T}_h}\le C \|\nabla w_{h}\|^2_{\mathcal{T}_h}+C\sum_{T\in\mathcal{T}_h}h_T^{-1}\|w_{h}-\widehat{w}^o_{h}\|^2_{\partial T}. \label{HDG-poin}
		\end{gather}
	\end{subequations}
\end{lemma}
The identity \eqref{beta1} can be obtained by integration by parts and  the proof of \eqref{beta2} follows from the fact $ \langle \bm\beta \cdot\bm{n}  \widehat w_h^o,  \widehat w_h^o \rangle_{\partial \mathcal{T}_h}=0$. For the last inequality \eqref{HDG-poin}, we refer to \cite[page 354]{MR3440284} for the proof.

Next, we define some seminorms:
\begin{definition}
	For all $(\bm q_h, y_h, \widehat y_h^o) \in \bm V_h\times W_h\times M_h^o$, define 
	\begin{subequations}
		\begin{align}
		\norm{(y_h,\widehat y_h^o)}_{W,w}^2&:=\varepsilon\|\nabla  y_h\|^2_{\mathcal{T}_h}+\|\tau^{1/2}( y_h-\widehat{y}_h^o)\|^2_{\partial\mathcal{T}_h}+\|\bar{\sigma}^{1/2}y_{h}\|^2_{\mathcal{T}_h},\label{def_Ww}\\
		\norm{(y_h,\widehat y_h^o)}_{W}^2&:=\varepsilon\|\nabla  y_h\|^2_{\mathcal{T}_h}+\|\tau^{1/2}( y_h-\widehat{y}_h^o)\|^2_{\partial\mathcal{T}_h}+\|(\beta_0+\bar{\sigma})^{1/2}y_{h}\|^2_{\mathcal{T}_h},\label{def_W}\\
		\norm{(\bm{q}_h,y_h,\widehat y_h^o)}_{w}^2&:=\varepsilon^{-1}\|\bm q_h\|^2_{\mathcal{T}_h}+\norm{(y_h,\widehat y_h^o)}_{W,w}^2,\label{def_trip_w}\\
		\norm{(\bm{q}_h,y_h,\widehat y_h^o)}^2&:=\varepsilon^{-1}\|\bm q_h\|^2_{\mathcal{T}_h}+\norm{(y_h,\widehat y_h^o)}_{W}^2.\label{def_trip}
		\end{align}
	\end{subequations}
\end{definition}

It is easy to see that the seminorm $\norm{(\cdot,\cdot)}_{W}$ is a norm since $\beta_0>0$, hence $\norm{(\cdot,\cdot,\cdot)}$ is also a norm. To prove the  seminorms  $\norm{(\cdot,\cdot)}_{W,w}$ and $\norm{(\cdot,\cdot,\cdot)}_w$  are norms, we just need to show $\norm{(\cdot,\cdot)}_{W,w}$  is a norm.
\begin{lemma} \label{norm} 
	$\norm{(\cdot,\cdot)}_{W,w}$  is a norm  for the space $W_h\times M_h^o$.
\end{lemma}
\begin{proof} 
	It is obvious that we only need to show that $\norm{(y_h,\widehat y_h^o)}_{W,w}=0$ implies $\widehat{y}_h^o|_{\Gamma}=y_h=0$. This is true because  $y_h$ is piecewise constant on $\mathcal{T}_h$ and $y_h=\widehat{y}_h^o$ on $\mathcal{E}_h$;  therefore, $y_h=\widehat{y}_h^o$ are constants. Since $\widehat{y}_h^o|_{\Gamma}=0$, we have  $y_h=\widehat{y}_h^o=0$.
\end{proof}

%
%


\begin{lemma} (Stability in weak norm) \label{stab-weak} 
	For all $(\bm{q}_h,y_h,\widehat y_h^o)\in \bm{V}_h\times W_h\times M_h^o$, the following stability results hold:
	\begin{subequations}
		\begin{align}
		\sup_{(\bm r_h,w_h,\widehat w_h^o)\in \bm V_h\times W_h\times M_h^o}\frac{\mathcal B_1(\bm{q}_h,y_h,\widehat y_h^o;\bm r_h,w_h,\widehat w_h^o)}{\norm{(\bm r_h,w_h,\widehat w_h^o)}_{w}}\ge C   \norm{(\bm{q}_h,y_h,\widehat y_h^o)}_{w}, \label{weak-stab-1} \\
		\sup_{(\bm r_h,w_h,\widehat w_h^o)\in \bm{V}_h\times W_h\times M_h^o}\frac{\mathcal B_2(\bm{q}_h,y_h,\widehat y_h^o;\bm r_h,w_h,\widehat w_h^o)}{\norm{(\bm r_h,w_h,\widehat w_h^o)}_{w}}\ge C   \norm{(\bm{q}_h,y_h,\widehat y_h^o)}_{w}\label{weak-stab-3}.
		\end{align}
	\end{subequations}
\end{lemma}
\begin{proof}
	We only prove the first inequality; the second can be obtained by the same argument. First, let $(\bm r_h,w_h,\widehat w_h^o) = (q_h,-y_h,-\widehat y^o_h)$ in the definition of $\mathcal B_1$ in \eqref{def_B1} to get
	\begin{align}
	\hspace{-1em}&\hspace{-1em}\mathcal B_1(\bm q_h,y_h,\widehat y^o_h;\bm q_h,-y_h,-\widehat y^o_h)\nonumber\\
	&=\varepsilon^{-1}\norm{\bm{q}_h}_{\mathcal{T}_h}^2+\langle \tau_1 (y_h-\widehat{y}^o_h), y_h-\widehat{y}^o_h\rangle_{\partial \mathcal{T}_h} \nonumber\\
	&\quad-( y_h,\bm\beta \nabla  y_h)_{\mathcal{T}_h}+\langle \widehat{y}^o_h,\bm \beta\cdot\bm n y_h\rangle_{\partial\mathcal{T}_h}
	+(\sigma y_h,y_h)_{\mathcal{T}_h}\nonumber\\
	&=\varepsilon^{-1}\norm{\bm{q}_h}_{\mathcal{T}_h}^2+\langle \tau_1 (y_h-\widehat{y}^o_h), y_h-\widehat{y}^o_h\rangle_{\partial \mathcal{T}_h} \nonumber\\
	&\quad -\frac 1 2 \langle \bm\beta \cdot\bm{n}  y_h, y_h \rangle_{\partial \mathcal{T}_h} +\langle \widehat{y}^o_h,\bm \beta\cdot\bm n y_h\rangle_{\partial\mathcal{T}_h}
	+((\sigma+\frac{1}{2}\nabla\cdot\bm\beta )y_h,y_h)_{\mathcal{T}_h}  & \textup{by} \  \eqref{beta1},\nonumber\\
	&=\varepsilon^{-1}\norm{\bm{q}_h}_{\mathcal{T}_h}^2+\langle \tau_1 (y_h-\widehat{y}^o_h), y_h-\widehat{y}^o_h\rangle_{\partial \mathcal{T}_h} \nonumber\\
	&\quad -\frac 1 2 \langle \bm\beta \cdot\bm{n}  (y_h - \widehat y_h^o), y_h-\widehat y_h^o \rangle_{\partial \mathcal{T}_h} 
	+((\sigma+\frac{1}{2}\nabla\cdot\bm\beta )y_h,y_h)_{\mathcal{T}_h}  & \textup{by} \  \eqref{beta2},\nonumber\\
	&=\varepsilon^{-1}\norm{\bm{q}_h}_{\mathcal{T}_h}^2+\norm{\sqrt{\tau} (y_h-\widehat{y}^o_h)}_{\partial \mathcal{T}_h}^2  +\|\bar{{\sigma}}^{1/2}y_h\|^2_{\mathcal{T}_h}   & \textup{by} \  \eqref{sigma}.\label{w1}
	\end{align}
	Next, let $(\bm r_h,w_h,\widehat w_h^o) = (\varepsilon\nabla y_h,0,0)$ to get
	\begin{align}\label{w2}
	\hspace{1em}&\hspace{-1em}\mathcal B_1(\bm{q}_h,y_h,\widehat y_h^o;\varepsilon\nabla y_h,0,0)\nonumber\\
	&=(\bm q_h,\nabla y_h)_{\mathcal{T}_h} -\varepsilon( y_h,\nabla \cdot\nabla y_h)_{\mathcal{T}_h} +\varepsilon\langle \widehat y^o_h,\nabla y_h\cdot\bm n\rangle_{\partial\mathcal T_h} \nonumber\\
	&=\varepsilon\|\nabla  y_h\|^2_{\mathcal{T}_h}+\varepsilon\langle \widehat y^o_h-y_h,\nabla y_h\cdot\bm n\rangle_{\partial\mathcal T_h}+(\bm q_h,\nabla y_h)_{\mathcal{T}_h}\nonumber\\
	&\ge \varepsilon\|\nabla  y_h\|^2_{\mathcal{T}_h} - \varepsilon^{1/2} h^{-1/2}   \|\widehat y^o_h-y_h\|_{\partial\mathcal T_h} \varepsilon^{1/2}\|\nabla y_h\|_{\mathcal T_h} - \|\bm q_h\|_{\mathcal T_h} \|\nabla y_h\|_{\mathcal{T}_h}\nonumber\\
	&\ge\frac{\varepsilon}{2}\|\nabla  y_h\|^2_{\mathcal{T}_h}-C_0(\|\tau^{1/2}( \widehat y^o_h-y_h)\|^2_{\partial\mathcal{T}_h}+\varepsilon^{-1}\|\bm{q}_h\|^2_{\mathcal{T}_h}),
	\end{align}
	where $C_0$ is a fixed positive constant. The definitions of $\norm{(\cdot,\cdot,\cdot)}$ in \eqref{def_trip} and $\norm{(\cdot,\cdot,\cdot)}_w$ in \eqref{def_trip_w} imply  
	\begin{align}
	\norm{(\varepsilon\nabla y_h,0,0)}=
	\norm{(\varepsilon\nabla y_h,0,0)}_w
	\le C\norm{(\bm q_h,y_h,\widehat{y}_h^o)}_w.\label{addition0}
	\end{align}
	Finally, we take  $(\bm r_h,w_h,\widehat w^o_h)=(\frac{1}{2}+C_0)(\bm q_h,-y_h,-\widehat y_h^o)+(\varepsilon\nabla y_h,0,0)$ to obtain
	\begin{align}
	\hspace{1em}&\hspace{-1em}\mathcal B_1(\bm q_h,y_h,\widehat y^o_h;\bm r_h,w_h,\widehat w^o_h)\nonumber\\
	&\ge \frac{1}{2}\norm{(\bm q_h,y_h,\widehat y_h^o)}^2_w       & \textup{by} \  \eqref{w1}  \ \textup{and} \  \eqref{w2},\nonumber\\
	&\ge C 
	\norm{(\bm q_h,y_h,\widehat y_h^o)}_w\norm{(\bm r_h,w_h,\widehat w_h^o)}_w & \textup{by} \  \eqref{addition0}. 
	\end{align}
	This  completes our proof. 
\end{proof}
For later use, by \eqref{addition0}, for any $(\bm q_h, y_h, \widehat y_h^o)\in \bm V_h\times W_h\times M_h^o$, we have 
\begin{align}
\norm{(\bm r_h,w_h,\widehat w_h^o)}\le C
\norm{(\bm q_h,y_h,\widehat y^o_h)}+\norm{(\varepsilon\nabla y_h,0,0)}\le C \norm{(\bm q_h,y_h,\widehat y^o_h)}
.\label{addition}
\end{align}

\begin{remark}
	The existence of a unique solution to the HDG discretization \eqref{HDG_full_discrete} of the optimality system now follows similarly to \cite{HuMateosSinglerZhangZhang1}; we omit the details. Also, to obtain the $L^2$ error estimates for the state $y_h$, \Cref{stab-weak} is not sufficient since the effective reaction term $\bar{\sigma}^{1/2}$ can equal zero at some points; therefore,  it is possible for the term $\|\bar{\sigma}^{1/2} y_h\|_{\mathcal{T}_h}$ in the definition of $\norm{(\bm q_h,y_h,\widehat{y}_h^o)}_w$ to equal zero for some $y_h$. Therefore, we need a refined analysis technique to derive a strong stability result that contains the norm  $\|y_h\|_{\mathcal{T}_h}$.
\end{remark}

\begin{theorem}\label{LBB} (Stability in strong norm) 	
	If assumptions  \textbf{(A1)} and  \textbf{(A2)} hold,  then there exists $h_0$, independent of $\varepsilon$,  such that the following stability  results hold: for all $(\bm{q}_h,y_h,\widehat y_h^o)\in \bm{V}_h\times W_h\times M_h^o$ with $h\le h_0$, 
	\begin{subequations}
		\begin{align}
		\sup_{(\bm r_h,w_h,\widehat w_h^o)\in \bm V_h\times W_h\times M_h^o}\frac{\mathcal B_1(\bm{q}_h,y_h,\widehat y_h^o;\bm r_h,w_h,\widehat w_h^o)}{\norm{(\bm r_h,w_h,\widehat w_h^o)} }\ge C   \norm{(\bm{q}_h,y_h,\widehat y_h^o)} , \label{stab-1} \\
		\sup_{(\bm r_h,w_h,\widehat w_h^o)\in \bm V_h\times W_h\times M_h^o}\frac{\mathcal B_2 (\bm{q}_h,y_h,\widehat y_h^o;\bm r_h,w_h,\widehat w_h^o)}{\norm{(\bm r_h,w_h,\widehat w_h^o)} }\ge C   \norm{(\bm{q}_h,y_h,\widehat y_h^o)}. \label{stab-3} 
		\end{align}
	\end{subequations}
\end{theorem}
\begin{proof} We only prove \eqref{stab-1}, and we split the proofs into two steps.
	
	{\bfseries{Step 1:}} Let $\psi\in W^{1,\infty}(\Omega)$ satisfy \eqref{condition-beta}. We take 
	\begin{align}\label{step1}
	(\bm r_h,w_h,\widehat w_h^o)=(\bm r_1,w_1,\widehat w_1) = ( \bm{0},-e^{-\psi}y_h,-e^{-\psi}\widehat y_h^o)
	\end{align}
	in the  definition of $\mathcal{B}_1$ in \eqref{def_B1} to obtain
	\begin{align}
	\hspace{1em}&\hspace{-1em}\mathcal B_1(\bm{q}_h,y_h,\widehat y_h^o;\bm r_1,w_1,\widehat w_1)\nonumber\\
	&=
	[	(e^{-\psi}y_h,\nabla \cdot\bm{q}_h)_{\mathcal{T}_h}-\langle e^{-\psi}\widehat y_h^o,\bm{q}_h\cdot\bm{n} \rangle_{\partial\mathcal{T}_h}] +\langle \tau_1(  y_h-\widehat{y}_h^o), e^{-\psi}y_h-e^{-\psi}\widehat y_h^o \rangle_{\partial\mathcal{T}_h}\nonumber\\
	&\quad+[-( y_h,\bm\beta\cdot \nabla  (e^{-\psi}y_h))_{\mathcal{T}_h}+\langle \widehat{y}^o_h,\bm \beta\cdot\bm n e^{-\psi} y_h\rangle_{\partial\mathcal{T}_h}]+(\sigma y_h,e^{-\psi}y_h)_{\mathcal{T}_h}\nonumber\\
	&=S_1+S_2+S_3+S_4.\nonumber
	\end{align}
	Next, we estimate $\{S_i\}_{i=1}^4$ term by term. First,
	\begin{align*}
	S_1&=-(\nabla  (e^{-\psi}y_h),\bm{q}_h)_{\mathcal{T}_h}-\langle e^{-\psi}(\widehat y_h^o-y_h),\bm{q}_h\cdot\bm{n} \rangle_{\partial\mathcal{T}_h} \nonumber\\
	&=
	-(y_h\nabla e^{-\psi}+e^{-\psi}\nabla y_h,\bm{q}_h)_{\mathcal{T}_h}-\langle e^{-\psi}(\widehat y_h^o-y_h),\bm{q}_h\cdot\bm{n} \rangle_{\partial\mathcal{T}_h} \nonumber\\
	&\le C (\varepsilon^{1/2}\|y_h\|_{\mathcal{T}_h}+\varepsilon^{1/2}\|\nabla  y_h\|_{\mathcal{T}_h}+\|\tau^{1/2}(\widehat y_h^o-y_h ) \|_{\partial\mathcal{T}_h})\varepsilon^{-1/2}\|\bm{q}_h\|_{\mathcal{T}_h} \nonumber\\
	&\le C (\varepsilon^{1/2}\|\nabla  y_h\|_{\mathcal{T}_h}+\|\tau^{1/2}(\widehat y_h^o-y_h ) \|_{\partial\mathcal{T}_h})\varepsilon^{-1/2}\|\bm{q}_h\|_{\mathcal{T}_h} \nonumber          & \textup{by} \  \eqref{HDG-poin}, \\
	&\le C\norm{(y_h,\widehat y_h^o)}_{W,w}\varepsilon^{-1/2}\|\bm{q}_h\|_{\mathcal{T}_h} \nonumber     & \textup{by} \  \eqref{def_Ww},\\
	&\le C\norm{(\bm{q}_h,y_h,\widehat y_h^o)}_{w}^2  & \textup{by} \  \eqref{def_trip_w}.
	\end{align*}
	Second, to estimate the term $S_3$, let  $\varphi = e^{-\psi}$ in \eqref{beta1} to obtain
	\begin{align*}
	-( y_h,\bm\beta\cdot \nabla  (e^{-\psi}y_h))_{\mathcal{T}_h} &= \frac{1}{2} ( \nabla\cdot\bm\beta e^{-\psi}y_h,y_h)_{\mathcal{T}_h} +\frac{1}{2}( y_h,e^{-\psi}y_h \bm\beta \cdot \nabla\psi )_{\mathcal{T}_h}\\
	&\quad -\frac 1 2 \langle y_h,\bm \beta\cdot\bm ne^{-\psi}y_h\rangle_{\partial\mathcal{T}_h}.
	\end{align*}
	Hence,     
	\begin{align*}
	S_3& = -( y_h,\bm\beta\cdot \nabla  (e^{-\psi}y_h))_{\mathcal{T}_h}+\langle \widehat{y}^o_h,\bm \beta\cdot\bm n e^{-\psi} y_h\rangle_{\partial\mathcal{T}_h} \\
	&= \frac{1}{2} ( \nabla\cdot\bm\beta e^{-\psi}y_h,y_h)_{\mathcal{T}_h} +\frac{1}{2}( y_h,e^{-\psi}y_h \bm\beta \cdot \nabla\psi )_{\mathcal{T}_h}\\
	&\quad -\frac 1 2 \langle y_h,\bm \beta\cdot\bm ne^{-\psi}y_h\rangle_{\partial\mathcal{T}_h} +\langle \widehat{y}^o_h,\bm \beta\cdot\bm n e^{-\psi} y_h\rangle_{\partial\mathcal{T}_h}\nonumber\\
	&= \frac{1}{2} (\nabla\cdot\bm\beta e^{-\psi} y_h,y_h)_{\mathcal{T}_h}+  \frac{1}{2} ( e^{-\psi}\bm\beta\cdot\nabla\psi y_h,y_h)_{\mathcal{T}_h}\\
	&\quad- \frac{1}{2} \langle\bm\beta\cdot\bm ne^{-\psi}  (y_h- \widehat y^o_h),y_h- \widehat y^o_h\rangle_{\partial\mathcal{T}_h}\nonumber    & \textup{by} \  \eqref{beta2}, \\
	&\ge \frac{1}{2}(\nabla\cdot\bm\beta e^{-\psi} y_h,y_h)_{\mathcal{T}_h}+( \beta_0e^{-\psi}y_h,y_h   )_{\mathcal{T}_h}      & \textup{by} \  \eqref{condition-beta},\\
	&\quad -\frac{1}{2}\langle\bm\beta\cdot\bm ne^{-\psi}  (y_h- \widehat y^o_h),y_h- \widehat y^o_h\rangle_{\partial\mathcal{T}_h}.
	\end{align*}
	Therefore, by the definition of $\tau $ in \eqref{def_tau}, there exist a positive constant $C_0$ such that  
	\begin{eqnarray}
	S_2+S_3+S_4&\ge& \|e^{-\psi/2}(\beta_0+\bar{\sigma})^{1/2}y_h\|^2_{\mathcal{T}_h}
	+ \|e^{-\psi/2}\tau^{1/2}(y_h-\widehat{y}_h^o)\|^2_{\partial\mathcal{T}_h} \nonumber\\
	&\ge& C_0\|(\beta_0+\bar{\sigma})^{1/2}y_h\|^2_{\mathcal{T}_h}.\nonumber
	\end{eqnarray}
	This implies that there exist positive constants $C_1$ and $C_2$ such that
	\begin{align}\label{BS1}
	\mathcal B_1(\bm{q}_h,y_h,\widehat y_h^o;\bm r_1,w_1,\widehat w_1)\ge  C_1\|(\beta_0+\bar{\sigma})^{1/2}y_h\|^2_{\mathcal{T}_h}-
	C_2\norm{(\bm{q}_h,y_h,\widehat y_h^o)}_{w}^2.
	\end{align}
	Moreover, we have 
	\begin{align}
	\norm{(\bm r_1,w_1,\widehat w_1)}^2&= \norm{( \bm{0},-e^{-\psi}y_h,-e^{-\psi}\widehat y_h^o)}\nonumber     & \textup{by} \  \eqref{step1},\\
	&=\varepsilon\|\nabla  (e^{-\psi}y_h)\|^2_{\mathcal{T}_h}+\|\tau^{1/2}e^{-\psi}( y_h-\widehat{y}_h^o)\|^2_{\partial\mathcal{T}_h}\nonumber\\
	&\quad+\|(\beta_0+\bar{\sigma})^{1/2}e^{-\psi}y_{h}\|^2_{\mathcal{T}_h}\nonumber  & \textup{by} \  \eqref{def_trip},\\
	&\le C \norm{(\bm{q}_h,y_h,\widehat y^o_h)}^2& \textup{by} \  \eqref{HDG-poin}, \label{ineqused}
	\end{align}
	where we used  $\nabla  (e^{-\psi}y_h)=y_h\nabla e^{-\psi}+e^{-\psi}\nabla y_h$ in \eqref{ineqused}.
	
	{\bfseries{Step 2:}} Let $R^o_k=\mathbb I-\Pi^o_k$ and $\widetilde{R}^{\partial}_k=\mathbb I-\widetilde{\Pi}^{\partial}_k$, where $\mathbb I$ is the identity operator.  We take 
	\begin{align}\label{secstep}
	(\bm r_h,w_h,\widehat w_h^o)=(\bm r_2,w_2,\widehat w_2) = ( \bm{0},R^{o}_k(e^{-\psi}y_h),\widetilde{R}^{\partial}_k(e^{-\psi}\widehat y_h^o))
	\end{align}	
	in the definition of $\mathcal{B}_1$ and use the orthogonality properties of $\Pi_k^o$ and  $\widetilde{\Pi}_k^{\partial}$, integration by parts,  and $\tau_1=\tau_2+\bm\beta\cdot\bm n$ to  get
	\begin{align}
	\hspace{1em}&\hspace{-1em}\mathcal B_1(\bm{q}_h,y_h,\widehat y_h^o;\bm r_2,w_2,\widehat w_2)\nonumber\\
	&=
	-(R^{o}_k(e^{-\psi}y_h),\nabla \cdot\bm{q}_h)_{\mathcal{T}_h}+\langle \widetilde{R}^{\partial}_k(e^{-\psi}\widehat y_h^o),\bm{q}_h\cdot\bm{n} \rangle_{\partial\mathcal{T}_h} \nonumber\\
	&\quad-\langle \tau_1(  y_h-\widehat{y}_h^o), R^{o}_k(e^{-\psi}y_h)-\widetilde{R}^{\partial}_k(e^{-\psi}\widehat y_h^o) \rangle_{\partial\mathcal{T}_h}\nonumber\\
	&\quad+( y_h,\bm\beta\cdot \nabla  R^{o}_k(e^{-\psi}y_h))_{\mathcal{T}_h}-\langle \widehat{y}^o_h,\bm \beta\cdot\bm n R^{o}_k(e^{-\psi} y_h)\rangle_{\partial\mathcal{T}_h}\label{needintr}\\
	&\quad-(\sigma y_h,R^{o}_k(e^{-\psi}y_h))_{\mathcal{T}_h}\nonumber.
	\end{align}
	The definitions of  $\Pi_k^o$ and  $\widetilde{\Pi}_k^{\partial}$ in \eqref{orth} imply
	\begin{align*}
	(R^{o}_k(e^{-\psi}y_h),\nabla \cdot\bm{q}_h)_{\mathcal{T}_h} = 0, \quad \textup{and}\quad \langle \widetilde{R}^{\partial}_k(e^{-\psi}\widehat y_h^o),\bm{q}_h\cdot\bm{n} \rangle_{\partial\mathcal{T}_h} =0.
	\end{align*}
	Next, integration by parts  gives
	\begin{align*}
	\hspace{1em}&\hspace{-1em}( y_h,\bm\beta\cdot \nabla  R^{o}_k(e^{-\psi}y_h))_{\mathcal{T}_h}\\
	&= \langle  \bm{\beta} \cdot \bm n y_h, R^{o}_k(e^{-\psi}y_h)\rangle_{\partial\mathcal T_h} - (\bm \beta \nabla y_h,R^{o}_k(e^{-\psi}y_h))_{\mathcal T_h}  -  (\nabla \cdot\bm \beta y_h, R^{o}_k(e^{-\psi}y_h))_{\mathcal T_h}\\
	&= \langle  \bm{\beta} \cdot \bm n (y_h-\widehat y_h^o), R^{o}_k(e^{-\psi}y_h)\rangle_{\partial\mathcal T_h}  +  \langle  \bm{\beta} \cdot \bm n \widehat y_h^o, R^{o}_k(e^{-\psi}y_h)\rangle_{\partial\mathcal T_h} \\
	&\quad - (\bm \beta \cdot\nabla y_h,R^{o}_k(e^{-\psi}y_h))_{\mathcal T_h}   -  (\nabla \cdot\bm \beta y_h, R^{o}_k(e^{-\psi}y_h))_{\mathcal T_h}\\
	&= \langle  \bm{\beta} \cdot \bm n (y_h-\widehat y_h^o), R^{o}_k(e^{-\psi}y_h) -  \widetilde{R}^{\partial}_k(e^{-\psi}\widehat y_h^o)  \rangle_{\partial\mathcal T_h}  +  \langle  \bm{\beta} \cdot \bm n (y_h-\widehat y_h^o), \widetilde{R}^{\partial}_k(e^{-\psi}\widehat y_h^o)  \rangle_{\partial\mathcal T_h}  \\
	&\quad +  \langle  \bm{\beta} \cdot \bm n \widehat y_h^o, R^{o}_k(e^{-\psi}y_h)\rangle_{\partial\mathcal T_h} - (\bm \beta\cdot \nabla y_h,R^{o}_k(e^{-\psi}y_h))_{\mathcal T_h}   -  (\nabla \cdot\bm \beta y_h, R^{o}_k(e^{-\psi}y_h))_{\mathcal T_h}.
	\end{align*}
	Using  $\tau_1 = \tau_2 + \bm \beta\cdot \bm n$ along with \eqref{needintr} and the above equalities give
	\begin{align}
	\hspace{1em}&\hspace{-1em}\mathcal B_1(\bm{q}_h,y_h,\widehat y_h^o;\bm r_2,w_2,\widehat w_2)\nonumber\\
	&=-\langle \tau_2(  y_h-\widehat{y}_h^o), R^{o}_k(e^{-\psi}y_h)-\widetilde{R}^{\partial}_k(e^{-\psi}\widehat y_h^o) \rangle_{\partial\mathcal{T}_h}\nonumber\\
	&\quad+\langle y_h-\widehat{y}^o_h,\bm \beta\cdot\bm n \widetilde{R}^{\partial}_k(e^{-\psi}\widehat y_h^o)\rangle_{\partial\mathcal{T}_h} -(\bm\beta\cdot\nabla  y_h, R^{o}_k(e^{-\psi}y_h))_{\mathcal{T}_h} \nonumber\\
	&\quad-((\sigma+\nabla\cdot\bm\beta) y_h,R^{o}_ke^{-\psi}y_h)_{\mathcal{T}_h}\nonumber\\
	&= T_1+T_2+T_3+T_4.\label{T0}
	\end{align}
	Before we estimate $\{T_i\}_{i=1}^4$, we first define ${R}^{\partial}_k=\mathbb I-{\Pi}^{\partial}_k$ and estimate  the following term:
	\begin{align}
	\hspace{1em}&\hspace{-1em}\|(R_{k}^o(e^{-\psi}y_h)-R_{k}^{\partial}(e^{-\psi}\widehat{y}^o_h))\|^2_{\partial \mathcal{T}_h}\nonumber\\
	&\le C \sum_{T\in\mathcal{T}_h}
	\left(
	\|R_{k}^o(e^{-\psi}y_h)\|^2_{\partial T}+\|R_{k}^{\partial}(e^{-\psi}\widehat{y}^o_h)\|^2_{\partial T}
	\right)\nonumber     \\
	&\le C \sum_{T\in\mathcal{T}_h}
	\left(
	h_T\|y_h\|^2_{T}
	+h_T^2\|y_h-\widehat{y}^o_h\|^2_{\partial T}
	\right)\nonumber   & \textup{by} \  \eqref{compo_0TE}-\eqref{compo_0EE},\\
	& \le Ch^2 \| y_h-\widehat{y}_h^o\|^2_{\partial\mathcal{T}_h}+Ch\|y_h\|^2_{\mathcal{T}_h} \label{tau-R}.
	\end{align}
	Therefore,
	\begin{align}\label{T1}
	|T_1|\le C h^2 \|\tau^{1/2}( y_h-\widehat{y}_h^o)\|^2_{\partial\mathcal{T}_h}+Ch\|(\beta_0+\bar{\sigma})^{1/2}y_h\|^2_{\mathcal{T}_h}.
	\end{align}
	Next, by the definition of  $R_k^o$,  we have
	\begin{align*}
	T_2+T_3+T_4&=(\bm\beta\cdot\nabla y_h,R_{k}^o(e^{-\psi}y_h))_{\mathcal{T}_h}+\langle y_h-\widehat{y}^o_h,\bm \beta\cdot\bm n \widetilde{R}^{\partial}_k(e^{-\psi}\widehat y_h^o)\rangle_{\partial\mathcal{T}_h}\\
	&\quad -((\sigma+\nabla\cdot\bm\beta )y_h,R_{k}^o(e^{-\psi}y_h))_{\mathcal{T}_h} \\
	&=(R_{0}^o(\bm\beta) \cdot\nabla y_h,R_{k}^o(e^{-\psi}y_h))_{\mathcal{T}_h}+ \langle y_h-\widehat{y}^o_h,\bm \beta\cdot\bm n \widetilde{R}^{\partial}_k(e^{-\psi}\widehat y_h^o)\rangle_{\partial\mathcal{T}_h} \\
	&\quad +((\sigma+\nabla\cdot\bm\beta )y_h,R_{k}^o(e^{-\psi}y_h)).
	\end{align*}
	Hence,
	\begin{align}
	T_2+T_3+T_4&\le C\left(\sum_{T\in\mathcal{T}_h}h_T^2\|\nabla y_h\|^2_{T}\right)^{1/2}h\|y_h\|_{\mathcal{T}_h} & \textup{by} \  \eqref{compo_0T}, \nonumber\\
	&\quad+C\left(\sum_{T\in\mathcal{T}_h}
	\|y_h-\widehat{y}^o_h\|^2_{\partial T}\right)^{1/2} \left(\sum_{T\in\mathcal{T}_h}h_T^2 \|y_h\|^2_{\partial T}\right)^{1/2} & \textup{by} \  \eqref{compo_0EE}, \nonumber\\
	&\quad+Ch(\|\bm\beta \|_{1,\infty}+\|\bar{\sigma}\|_{0,\infty})\|y_h\|^2_{\mathcal{T}_h}  & \textup{by} \  \eqref{compo_0T},\nonumber\\
	&\le Ch\|(\beta_0+\bar{\sigma})^{1/2}y_h\|^2_{\mathcal{T}_h}+\|\tau^{1/2}( y_h-\widehat{y}_h^o)\|^2_{\partial\mathcal{T}_h}.\label{T234}
	\end{align}
	
	From \eqref{T0}, $(\ref{T1})$ and $(\ref{T234})$, we  get
	\begin{align}\label{BS2}
	\mathcal B_1(\bm{q}_h,y_h,\widehat y_h^o;\bm r_2,w_2,\widehat w_2^o)\ge-C_3h\|(\beta_0+\bar{\sigma})^{1/2}y_h\|^2_{\mathcal{T}_h}
	-C_4\norm{(\bm{q}_h,y_h,\widehat y_h^o)}_{w}^2.
	\end{align}
	Using  \eqref{tau-R}, we have
	\begin{align}
	\norm{(\bm r_2,w_2,\widehat w_2^o)}^2&=\norm{( \bm{0},R^{o}_k(e^{-\psi}y_h),\widetilde{R}^{\partial}_k(e^{-\psi}\widehat y_h^o))}\nonumber\\
	& = \varepsilon\|\nabla  R^{o}_k(e^{-\psi}y_h)\|^2_{\mathcal{T}_h}+\|(\beta_0+\bar{\sigma})^{1/2}e^{-\psi}y_{h}\|^2_{\mathcal{T}_h} \nonumber\\
	&\quad+\|\tau^{1/2}(R^{o}_k(e^{-\psi} y_h)-\widetilde{R}^{\partial}_k(e^{-\psi}\widehat{y}_h^o))\|^2_{\partial\mathcal{T}_h}\nonumber\\
	&\le C \norm{(\bm{q}_h,y_h,\widehat y_h^o)}^2 & \textup{by} \  \eqref{HDG-poin}.
	\end{align}
	By  \eqref{addition}, there exists a $(\bm r_0,w_0,\widehat w^o_0)$ $\in \bm V_h\times W_h\times M_h^o$ such that
	\begin{subequations}\label{BS3}
		\begin{gather}
		\mathcal B_1(\bm q_h,y_h,\widehat{y}_h^o;\bm r_0,w_0,\widehat w^o_0)\ge \norm{(\bm q_h,y_h,\widehat{y}_h^o)}^2_w,\\
		\norm{(\bm r_0,w_0,\widehat w^o_0)}\le C
		\norm{(\bm q_h,y_h,\widehat{y}_h^o)}.
		\end{gather}
	\end{subequations} 
	Take $h$ small enough so that we have  $C_3h\le C_1/2$.  Set  $C_*=C_1+C_2+C_4$ and 
	\begin{align}
	(\bm r_h,w_h,\widehat w^o_h) = C_*(\bm r_0,w_0,\widehat w^o_0)+(\bm r_1,w_1,\widehat w^o_1)+(\bm r_2,w_2,\widehat w^o_2).
	\end{align}
	By \eqref{BS1}, \eqref{BS2} and \eqref{BS3}, we  get
	\begin{align*}
	\mathcal B_1(\bm q_h,y_h,\widehat y^o_h;\bm{r}_h,w_h,\widehat w^o_h)&\ge C
	\norm{(\bm q_h,y_h,\widehat y^o_h)}^2
	\ge C\norm{(\bm q_h,y_h,\widehat y^o_h)}\cdot
	\norm{(\bm r_h,w_h,\widehat w_h^o)},
	\end{align*}
	which implies \eqref{stab-1}. 
\end{proof}

\section{Error analysis}
\label{erroranalysis}
Next, we perform a convergence analysis for the convection dominated Dirichlet boundary control problem.

\subsection{Assumptions and main result}

Throughout, we assume $\Omega$ is a bounded convex polyhedral domain.  Therefore, in the 2D  case  the largest interior angle $\omega$ satisfies $\pi/3\le \omega<\pi$. Moreover, we assume the velocity vector field $ \bm \beta $ and  $\sigma$ satisfy
\begin{equation}\label{eqn:beta_assumptions2}
\bm \beta \in [C(\overline{\Omega})]^d,  \nabla\cdot\bm{\beta}\in L^\infty(\Omega),      \sigma+\frac 1 2\nabla \cdot \bm \beta \geq 0,  \nabla \nabla \cdot \bm{\beta}\in[L^2(\Omega)]^d,  \sigma\in L^{\infty}(\Omega)\cap H^{1}(\Omega).
\end{equation}
We assume the solution  has the following regularity properties:
\begin{subequations}\label{eqn:regularity2}
	\begin{gather}
	y \in H^{r_y}(\Omega),  \quad  z \in H^{r_z}(\Omega),  \quad  \bm q \in [ H^{r_{\bm q}}(\Omega) ]^d ,  \quad  \bm p \in [H^{r_{\bm p}}(\Omega)]^d,\\
	r_y \ge 1,  \quad  r_z \ge 2,  \quad  r_{\bm q} \ge  0,  \quad  r_{\bm p} \ge 1.\label{eqn:s_rates_ineq}
	\end{gather}
\end{subequations}
In the 2D case, \Cref{MT210} guarantees this condition is satisfied.

It is worthwhile to mention that if $ \bm q $ has a well-defined boundary trace in $ L^2(\Gamma) $, i.e., $ r_{\bm q} > 1/2 $, then we refer to this as the high regularity case for the boundary control problem; otherwise, if $ r_{\bm q} \in [0,1/2] $, then we say this is the low regularity case.  In 2D, by  \Cref{MT210}, if $ y_d \in H^{t^*}(\Omega) $ for some $ t^* \in (1/2,1) $, and $ \pi/3 \le \omega < 2\pi/3 $, then we are guaranteed to be in the high regularity case.  However, if one of the above assumptions concerning $ y_d $ or $ \omega $ is not satisfied, then $ \bm q $ is no longer guaranteed to have a well-defined boundary trace.

For the \emph{diffusion dominated}  boundary control  problem, we gave a rigorous error analysis of a different HDG method for  the high regularity case in  \cite{HuShenSinglerZhangZheng_HDG_Dirichlet_control1,HuMateosSinglerZhangZhang1} and for the  low regularity case in \cite{HuMateosSinglerZhangZhang2}. In this work, we are interested  in the \emph{convection dominated} case. However, existing numerical analysis works for convection dominated diffusion PDEs only consider the high regularity case; see, e.g. \cite{MR2485457,MR3342199}. To the best of our knowledge, there is no existing error analysis work on convection dominated PDEs with low regularity solutions.

We now state our main convergence result.
\begin{theorem}\label{main-conver} 
	Let  $s_{y} = \min\{r_{y}, k+1\}$,  $ s_{z} = \min\{r_{z}, k+1\}$,  $(u,y,z)$ and $(u_h,y_h,\\ z_h)$ be the solutions of \eqref{eq_adeq} and \eqref{HDG_discrete2}, respectively. If assumptions  \textbf{(A1)}-\textbf{(A3)} hold,  then there exists $h_0$, independent of $\varepsilon$,  such that for all $h\le h_0$ we have 
	\begin{align*}
	\|u-u_h\|_{\mathcal E_h^\partial}&\le C
	\left(h^{s_{y}-1/2}\|y\|_{s_y}+h^{s_{z}-1/2}\|z\|_{s_z}+\delta(s_y)\varepsilon^{1/2}h\|\Delta y\|_{\mathcal{T}_h}\right),\\
	\|y - y_h \|_{\mathcal{T}_h}&\le C\left(h^{s_{y}-1/2}\|y\|_{s_y}+ h^{s_{z}-1/2}\|z\|_{s_z}+\delta(s_y)\varepsilon^{1/2}h\|\Delta y\|_{\mathcal{T}_h}\right),\\
	\|z-z_h \|_{\mathcal{T}_h}&\le C\left(h^{s_{y}-1/2}\|y\|_{s_y}+ h^{s_{z}-1/2}\|z\|_{s_z}+\delta(s_y)\varepsilon^{1/2}h\|\Delta y\|_{\mathcal{T}_h}\right),
	\end{align*}
	where $\delta (t) =1$ if $t\le 3/2$,  otherwise $\delta (t) =0$. 
\end{theorem}
\begin{remark}
	If $s_y \le3/2$, then we  have $\|\Delta y\|_{\mathcal T_h}$ in the error estimates. This term is finite by \Cref{MT210}.
\end{remark}

Specializing to the 2D case gives the following result:
\begin{corollary}\label{cor:main_result1}
	Suppose $ d = 2 $, $ f = 0 $, $ y_d \in H^{t^*}(\Omega) $ for some $ t^* \in [0,1) $ and  assumptions  \textbf{(A1)}-\textbf{(A3)} hold.  Let $ \pi/3\le \omega<\pi $ be the largest interior angle of $\Gamma$, and let $ r > 0 $ satisfy
	$$
	r \leq  r_d := \frac{1}{2} + t^* \in [1/2,3/2),  \quad  \mbox{and}  \quad  r < r_{\Omega} := \min\left\{ \frac{3}{2}, \frac{\pi}{\omega} - \frac{1}{2} \right\} \in (1/2, 3/2].
	$$
	If $ k = 1 $,  then there exists $h_0$, independent of $\varepsilon$,  such that for all $h\le h_0$ we have 
	\begin{align*}
	\norm{u-u_h}_{\varepsilon_h^\partial}& \le C h^{r} (\norm{y}_{H^{r+1/2}(\Omega)} + \norm{z}_{H^{r+3/2}(\Omega)} + \delta(r)\varepsilon^{1/2}h\|\Delta y\|_{\mathcal{T}_h}),\\
	\norm{y-y_h}_{\mathcal T_h}& \le C h^{r} (\norm{y}_{H^{r+1/2}(\Omega)} + \norm{z}_{H^{r+3/2}(\Omega)} +\delta(r)\varepsilon^{1/2}h\|\Delta y\|_{\mathcal{T}_h}),\\
	\norm {z - z_h}_{\mathcal T_h}   &  \le C  h^{r} (\norm{y}_{H^{r+1/2}(\Omega)} + \norm{z}_{H^{r+3/2}(\Omega)} + \delta(r)\varepsilon^{1/2}h\|\Delta y\|_{\mathcal{T}_h}).
	\end{align*}
	Furthermore, if $ k = 0 $,  then there exists $h_1$, independent of $\varepsilon$, such that for all $h\le h_1$ we have 
	\begin{align*}
	\norm{u-u_h}_{\varepsilon_h^\partial}&\le C h^{1/2} (\norm{y }_{H^{1}(\Omega)}+\norm{z }_{H^{1}(\Omega)}+\delta(r)\varepsilon^{1/2}h\|\Delta y\|_{\mathcal{T}_h}),\\
	\norm{y-y_h}_{\mathcal T_h}&\le C h^{1/2} (\norm{y }_{H^{1}(\Omega)}+\norm{z }_{H^{1}(\Omega)}+\delta(r)\varepsilon^{1/2}h\|\Delta y\|_{\mathcal{T}_h}),\\
	\norm{z-z_h}_{\mathcal T_h}&\le C h^{1/2} (\norm{y }_{H^{1}(\Omega)}+\norm{z }_{H^{1}(\Omega)}+\delta(r)\varepsilon^{1/2}h\|\Delta y\|_{\mathcal{T}_h}).
	\end{align*}
	
\end{corollary}

Similar to \cite{HuShenSinglerZhangZheng_HDG_Dirichlet_control1,HuMateosSinglerZhangZhang1}, the convergence rates are optimal for the control when $k=1$ and suboptimal when $k=0$. However,  if $ y_d \in L^2(\Omega) $, then $ u \in H^{1/2}(\Gamma) $ only and the convergence rate for the control is optimal when $k=0$.

\subsection{Proof of \Cref{main-conver}}
We introduce an auxiliary problem with the approximate control $u_h$ in the HDG discretized  optimality system \eqref{HDG_full_discrete_a} replaced by a projection of the exact optimal control, and split the proof into seven steps.

We first bound the error between the solution of the optimality system \eqref{mixed_a1}-\eqref{mixed_d1} and
$(\bm q_h(u),y_h(u),\widehat y^o_h(u),\bm p_h(u),z_h(u),\widehat z^o_h(u))\in [\bm{V}_h\times W_h\times M^o_h ]^2$ satisfying the auxiliary problem
\begin{subequations}\label{Co}
	\begin{align}
	\mathcal B_1(\bm q_h(u),y_h(u),\widehat y_h(u);\bm r_1,w_1,\widehat w_1^o)&=-( f, w_1)_{\mathcal{T}_h}-\langle \Pi^{\partial}_k u, \tau_2 w_1+\bm r_1\cdot\bm{n}\rangle_{\mathcal{E}_h^{\partial}},\label{Co1}\\
	\mathcal B_2(\bm p_h(u),z_h(u),\widehat z_h(u);\bm r_2,w_2,\widehat w_2^o)&=-(y_h(u)-y_d,w_2)_{\mathcal{T}_h}\label{Co2},
	\end{align}
	for all  $\left(\bm r_1,w_1,\widehat{w}^o_1,\bm r_2,w_2,\widehat{w}^o_2\right)\in [\bm V_h\times W_h\times M^{o}_h]^2$.
\end{subequations}

\subsubsection{Step 1: errors between the auxiliary problem \eqref{Co} and the  continuous problem \eqref{mixed}}
\label{part I}

\begin{lemma}\label{main_lemma_step1} 
	Let $(\bm{q},y,\bm{p},z,u)$ be the solution of \eqref{mixed}. Then for all  
	$(\bm r_1,w_1,\widehat{w}^o_1,\bm r_2, \\  w_2,  \widehat{w}^o_2)\in [\bm V_h\times W_h\times M^{o}_h]^2$, we have 
	\begin{subequations}\label{Co-error}
		\begin{align}
		&\mathcal B_1(\bm{\Pi}^o_k\bm{q},\Pi^{o}_ky,\widetilde{\Pi}^{\partial}_ky; \bm r_1,w_1,\widehat w_1^o)\nonumber\\
		&\qquad=-(f,w_1)_{\mathcal{T}_h}-\langle \Pi^{\partial}_ku,\tau_2 w_1+\bm r_1\cdot\bm n\rangle_{\mathcal{E}_h^{\partial}}
		+E_1(\bm{q},y;w_1,\widehat{w}_1^o),\label{Co-error1}\\
		&\mathcal B_2(\bm{\Pi}^o_k\bm p,\Pi^{o}_kz,\widetilde{\Pi}^{\partial}_k z;\bm r_2,w_2,\widehat w_2^o)=-(y-y_d,w_2)_{\mathcal{T}_h}+E_2(\bm{p},z;w_2,\widehat{w}_2^o),\label{Co-error2}
		\end{align}
	\end{subequations}
	where
	\begin{eqnarray}
	E_1(\bm{q},y;w_1,\widehat w_1^o)&=&( \Pi^{o}_ky-y,\bm\beta\cdot\nabla w_1)_{\mathcal{T}_h} -(\sigma (\Pi^{o}_k-\mathbb I)y,w_1)_{\mathcal{T}_h}\nonumber\\
	&&+\langle \widehat w_1^o-w_1,(\bm{\Pi}^o_k\bm{q}-\bm{q})\cdot\bm{n} \rangle_{\partial\mathcal{T}_h}\nonumber\\
	&&+\langle y-\Pi_k^{\partial}y,\bm \beta\cdot\bm n (w_1-\widehat w^o_1)\rangle_{\partial\mathcal{T}_h}\nonumber\\
	&&-\langle \tau_1(  \Pi^{o}_ky-\Pi^{\partial}_ky), w_1-\widehat w^o_1 \rangle_{\partial\mathcal{T}_h},\\
	E_2(\bm{p},z;w_2,\widehat{w}_2^o)&=&-( \Pi^{o}_kz-z,\bm\beta\cdot\nabla w_2)_{\mathcal{T}_h} -(\sigma (\Pi^{o}_k-\mathbb I)(z+\nabla\cdot\bm\beta),w_2)_{\mathcal{T}_h}\nonumber\\
	&&+\langle \widehat{w}_2^o-w_2,(\bm{\Pi}^o_k\bm{p}-\bm{p})\cdot\bm{n} \rangle_{\partial\mathcal{T}_h}\nonumber\\
	&&-\langle z-\Pi_k^{\partial}z,\bm \beta\cdot\bm n ( w_2-\widehat w^o_2)\rangle_{\partial\mathcal{T}_h}\nonumber\\
	&&-\langle \tau_2(  \Pi^{o}_kz-\Pi^{\partial}_kz), w_2-\widehat{w}_2^o \rangle_{\partial\mathcal{T}_h}.
	\end{eqnarray}
\end{lemma}
\begin{proof} 
	We only give a proof of \eqref{Co-error1}. By the definition of $\mathcal B_1$ in  \eqref{def_B1}, one gets
	\begin{align*}
	\hspace{1em}&\hspace{-1em} \mathcal B_1(\bm{\Pi}^o_k\bm{q},\Pi^{o}_ky,\widetilde{\Pi}^{\partial}_ky;\bm r_1,w_1,\widehat w_1^o)\nonumber\\
	&=\varepsilon^{-1}(\bm{\Pi}^o_k\bm q,\bm r_1)_{\mathcal{T}_h} -( \Pi^{o}_ky,\nabla \cdot\bm r_1)_{\mathcal{T}_h} +\langle \widetilde{\Pi}^{\partial}_ky,\bm r_1\cdot\bm n\rangle_{\partial\mathcal T_h}\nonumber\\
	&\quad -(w_1,\nabla \cdot\bm{\Pi}^o_k\bm{q})_{\mathcal{T}_h}+\langle \widehat w_1^o,\bm{\Pi}^o_k\bm{q}\cdot\bm{n} \rangle_{\partial\mathcal{T}_h}
	-\langle \tau_1(  \Pi^{o}_ky-\widetilde{\Pi}^{\partial}_ky), w_1-\widehat w_1^o \rangle_{\partial\mathcal{T}_h}\nonumber\\
	&\quad+( \Pi^{o}_ky,\bm\beta\cdot \nabla  w_1)_{\mathcal{T}_h}-\langle \widetilde{\Pi}^{\partial}_ky,\bm \beta\cdot\bm n  w_1\rangle_{\partial\mathcal{T}_h}-(\sigma \Pi_k^oy,w_1)_{\mathcal{T}_h}.
	\end{align*}
	By the orthogonality properties  of $\bm{\Pi}^o_k$, $\Pi^{o}_k$, $\widetilde{\Pi}^{\partial}_k$,
	and the fact $y=u$ on $\mathcal{E}_h^{\partial}$, we have 
	\begin{align*}
	\hspace{1em}&\hspace{-1em}\mathcal B_1(\bm{\Pi}^o_k\bm{q},\Pi^{o}_ky,\widetilde{\Pi}^{\partial}_ky;\bm r_1,w_1,\widehat w_1^o)\nonumber\\
	&=\varepsilon^{-1}(\bm q,\bm r_1)_{\mathcal{T}_h} -( y,\nabla \cdot\bm r_1)_{\mathcal{T}_h} +\langle y,\bm r_1\cdot\bm n\rangle_{\partial\mathcal T_h}
	-\langle  \Pi_k^{\partial}u,\bm r_1 \cdot\bm n \rangle_{\mathcal{E}_h^{\partial}} \nonumber\\
	&\quad +(\nabla w_1,\bm{q})_{\mathcal{T}_h}+\langle \widehat w_1^o-w_1,\bm{\Pi}^o_k\bm{q}\cdot\bm{n} \rangle_{\partial\mathcal{T}_h}
	-\langle \tau_1(  \Pi^{o}_ky-\widetilde{\Pi}^{\partial}_ky), w_1-\widehat{w}_1^o \rangle_{\partial\mathcal{T}_h}\nonumber\\
	&\quad+( \Pi^{o}_ky-y,\bm\beta\cdot \nabla  w_1)_{\mathcal{T}_h}-\langle {\Pi}^{\partial}_ky-y,\bm \beta\cdot\bm n  w_1\rangle_{\partial\mathcal{T}_h}-(\sigma (\Pi_k^oy-y),w_1)_{\mathcal{T}_h}\nonumber\\
	&\quad+(y,\bm\beta\cdot \nabla  w_1)_{\mathcal{T}_h}-\langle y,\bm \beta\cdot\bm n  w_1\rangle_{\partial\mathcal{T}_h}-(\sigma y,w_1)_{\mathcal{T}_h}
	+\langle \Pi_k^{\partial}u,\bm\beta\cdot\bm n w_1 \rangle_{\mathcal E_h^\partial}.
	\end{align*}
	By integration by parts,
	and the fact $\langle \widehat{w}_1^o,\bm{q}\cdot\bm{n} \rangle_{\partial\mathcal{T}_h}=0$, we arrive at
	\begin{align*}
	\hspace{1em}&\hspace{-1em} \mathcal B_1(\bm{\Pi}^o_k\bm{q},\Pi^{o}_ky,\widetilde{\Pi}^{\partial}_ky;\bm r_1,w_1,\widehat w_1^o)\nonumber\\
	&=\varepsilon^{-1}(\bm q,\bm r_1)_{\mathcal{T}_h} +( \nabla y,\bm r_1)_{\mathcal{T}_h} -\langle \Pi^{\partial}_ku,\bm r_1\cdot\bm n\rangle_{\mathcal{E}_h^{\partial}} -(w_1,\nabla\cdot\bm{q})_{\mathcal{T}_h}\nonumber\\
	&\quad
	+\langle \widehat w_1^o-w_1,(\bm{\Pi}^o_k\bm{q}-\bm{q})\cdot\bm{n} \rangle_{\partial\mathcal{T}_h}
	-\langle \tau_1(\Pi^{o}_ky-\Pi^{\partial}_ky), w_1-\widehat{w}_1^o \rangle_{\partial\mathcal{T}_h}\nonumber\\
	&\quad-\langle \tau_1 \Pi^{\partial}_ku,w_1\rangle_{\mathcal{E}_h^{\partial}}+( \Pi^{o}_ky-y,\bm\beta\cdot\nabla w_1)_{\mathcal{T}_h}
	+\langle y-  \Pi^{\partial}_ky,\bm \beta\cdot\bm n  w_1\rangle_{\partial\mathcal{T}_h} \nonumber\\
	&\quad -(\nabla\cdot (\bm\beta y), w_1)_{\mathcal{T}_h}+\langle \bm\beta\cdot\bm{n}\Pi^{\partial}_ku,w_1\rangle_{\mathcal{E}_h^{\partial}}-(\sigma y,w_1)_{\mathcal{T}_h}-(\sigma(\Pi^{o}_k-\mathbb I)y,w_1)_{\mathcal{T}_h}.\nonumber
	\end{align*}
	Then by the facts $\varepsilon^{-1}\bm{q}=-\nabla y$ and  $\nabla\cdot\bm{q}+\nabla\cdot(\bm\beta y)+\sigma y=f$, we have 
	\begin{align}
	\hspace{1em}&\hspace{-1em}\mathcal B_1(\bm{\Pi}^o_k\bm{q},\Pi^{o}_ky,\widetilde{\Pi}^{\partial}_ky;\bm r_1,w_1,\widehat w_1^o)\nonumber\\
	&=-(f,w_1)_{\mathcal{T}_h}-\langle \Pi^{\partial}_k u,\tau_2 w_1+\bm r_1\cdot\bm n\rangle_{\mathcal{E}_h^{\partial}}\nonumber\\
	&\quad+(\Pi^{o}_ky-y,\bm\beta\cdot\nabla w_1)_{\mathcal{T}_h} -(\sigma (\Pi^{o}_k-\mathbb I)y,w_1)_{\mathcal{T}_h}\nonumber\\
	&\quad+\langle \widehat{w}_1^o-w_1,(\bm{\Pi}^o_k\bm{q}-\bm{q})\cdot\bm{n} \rangle_{\partial\mathcal{T}_h}
	-\langle \tau_1(  \Pi^{o}_ky-\Pi^{\partial}_ky), w_1-\widehat{w}_1^o \rangle_{\partial\mathcal{T}_h}\nonumber\\
	&\quad+\langle y-  \Pi^{\partial}_ky,\bm \beta\cdot\bm n ( w_1-\widehat w_1^o)\rangle_{\partial\mathcal{T}_h},\label{betan}
	\end{align}
	where we used 
	$\langle y-  \Pi^{\partial}_ky,\bm \beta\cdot\bm n  \widehat w_1^o\rangle_{\partial\mathcal{T}_h}=0$ in \eqref{betan}.
\end{proof}

By \eqref{Co-error} and \eqref{Co} we have the following error equations.
\begin{lemma} Let $\left(\bm q, y,\bm p, z,u\right)$ and
	$(\bm q_h(u),y_h(u),\widehat{y}^o_h(u),\bm p_h(u),z_h(u),\widehat{z}^o_h(u))\in [ \bm V_h\\ \times W_h\times M^{o}_h]^2$ be
	the solutions of \eqref{mixed} and \eqref{Co}, respectively.  Then
	for all   $(\bm r_1,w_1,\widehat{w}^o_1,\bm r_2,\\
	w_2,\widehat{w}^o_2)\in [\bm V_h\times W_h\times M^{o}_h]^2$, we have 
	\begin{subequations}\label{error00}
		\begin{align}
		\mathcal B_1(\bm{\Pi}^o_k\bm{q}-\bm q_h(u),\Pi^{o}_ky-y_h(u),\widetilde{\Pi}^{\partial}_ky-\widehat{y}_h^o(u);\bm r_1,w_1,\widehat w_1^o)&=E_1(\bm{q},y;w_1,\widehat{w}_1^o),\label{error001}\\
		\mathcal B_2(\bm{\Pi}^o_k\bm p-\bm p_h(u),\Pi^{o}_kz-z_h(u),\widetilde{\Pi}^{\partial}_k z-\widehat{z}_h^o(u);\bm r_2,w_2,\widehat w_2^o)&=-(y-y_h(u),w_2)_{\mathcal{T}_h} \nonumber\\
		&\quad +E_2(\bm{p},z;w_2,\widehat{w}_2^o),\label{error002}
		\end{align}
		where  $E_1$ and $E_2$ are defined in \Cref{main_lemma_step1}.
	\end{subequations}
\end{lemma}

\begin{lemma} \label{energyforE1andE2}
	Let $\left(\bm q, y,\bm p, z\right)$ be the solution of  \eqref{mixed}. Then for all  $(w_1,w_2,\widehat w_1^o,\widehat w_2^o)\in [W_h\times M_h^o]^2$, we have 
	\begin{eqnarray}
	|E_1(\bm{q},y;w_1,\widehat w_1^o)|&\le C&\left(h^{s_y-1/2}\|y\|_{s_y}+\delta(s_y)\varepsilon^{1/2}h\|\Delta y\|_{\mathcal{T}_h}\right)\norm{(w_1,\widehat w_1^o)}_W,\label{es1}\\
	|E_2(\bm{p},z;w_2,\widehat{w}_2^o)|&\le C& h^{s_z-1/2}\|z\|_{s_z}\norm{(w_2,\widehat{w}_2^o)}_W. \label{es2}
	\end{eqnarray}
\end{lemma}
\begin{proof} Since the proof for \eqref{es2} is similar to the proof of  \eqref{es1}, we only prove \eqref{es1}. To simplify notation, we write $E_1$ from \Cref{main_lemma_step1} as $E_1(\bm{q},y;v_h,\widehat{v}_h^o)=\sum_{i=1}^5R_i$. For the term $R_1$, 
	since $((\mathbb I-\Pi^{o}_k)y,\bm{\Pi}_0^o\bm{\beta}\cdot\nabla  w_1)_{\mathcal{T}_h}=0$, we get
	\begin{align*}
	|R_1|&= |((\mathbb I-\Pi^{o}_k)y,(\bm{\beta}-\bm{\Pi}_0^o\bm{\beta})\cdot\nabla  w_1)_{\mathcal{T}_h}|\nonumber\\
	&\le C  h^{s_y}|\bm{\beta}|_{1,\infty}\|y\|_{s_y}\left(\sum_{T\in\mathcal{T}_h}h_T^2\|\nabla w_1\|^2_{T}\right)^{1/2}\nonumber\\
	&\le C  h^{s_y} \|y\|_{s_y}\|w_1\|_{\mathcal{T}_h}    & \textup{by} \  \eqref{stab_7},\\
	&\le C h^{s_y}\|y\|_{s_y}\norm{(w_1,\widehat{w}_1^o)}_W  & \textup{by} \  \eqref{def_W}.
	\end{align*}
	For the term $R_2$, by a direct estimate,  we get
	\begin{align*}
	|R_2|
	&\le C \|\sigma\|^{1/2}_{0,\infty}\|(\mathbb I-\Pi^{o}_k)y\|_{\mathcal{T}_h}\|\sigma^{1/2} w_1\|_{\mathcal{T}_h}\nonumber\\
	&\le C h^{s_y}\|\sigma\|^{1/2}_{0,\infty}\|y\|_{s_y}\norm{(w_1,\widehat{w}_1^o)}_W\nonumber & \textup{by} \  \eqref{def_W},\\
	&\le C h^{s_y}(\beta_0^{1/2}+\sigma_0^{1/2})\|y\|_{s_y}\norm{(w_1,\widehat{w}_1^o)}_W\nonumber\\
	&\le C h^{s_y}\|y\|_{s_y}\norm{(w_1,\widehat{w}_1^o)}_W.
	\end{align*}
	For the term $R_3$, we need a refined analysis since this term relates to the boundary trace of  the gradient of $y$.  Below, we use  $\bm q=-\varepsilon \nabla y$ and $\varepsilon< \min_{T\in \mathcal T_h}\{h_T\}$.
	
	If $s_y> 3/2$,   we have
	\begin{align*}
	|R_3|&=\varepsilon|\langle \bm{n} \cdot( \nabla y-\bm{\Pi}_k^o\nabla y),w_1-\widehat w_{1}^o\rangle_{\partial \mathcal{T}_h}|
	\nonumber\\
	&\le C  h^{s_y-1}\varepsilon^{1/2}\|y\|_{s_y}\left(\sum_{T\in\mathcal{T}_h}\varepsilon h_T^{-1} \|w_1-\widehat w_{1}^o\|^2_{\partial T}\right)^{1/2}\nonumber\\
	&\le C h^{s_{y}-1/2}\|y\|_{s_{y}}\norm{(w_1,\widehat{w}_1^o)}_W.
	\end{align*}
	If $s_y\le3/2$, use
	$\langle \bm n\cdot\nabla y,\widehat w_1^o\rangle_{\partial\mathcal{T}_h}=0$
	and integration by parts to obtain
	\begin{eqnarray}
	|R_3|&=&\varepsilon|\langle \bm{n} \cdot \nabla y,w_1\rangle_{\partial \mathcal{T}_h}
	-\langle \bm{n}\cdot\bm{\Pi}_k^o\nabla y,w_1-\widehat w_1^o\rangle_{\partial \mathcal{T}_h}|
	\nonumber\\
	&=&\varepsilon|(\Delta y, w_1)_{\mathcal{T}_h}+(\bm{\Pi}_k^o\nabla y,\nabla  w_1)_{\mathcal{T}_h}-\langle \bm{n}\cdot\bm{\Pi}_k^o\nabla y,w_1-\widehat{w}_{1}^o\rangle_{\partial \mathcal{T}_h}|.\nonumber
	\end{eqnarray}
	We use integration by parts again, and  also \eqref{I*2} and \eqref{I*1} to get
	\begin{eqnarray}
	|R_3|
	&=&\varepsilon|(\Delta y, w_1)_{\mathcal{T}_h}-(\nabla\cdot\bm{\Pi}_k^o\nabla y,w_1)_{\mathcal{T}_h}+\langle \bm{n}\cdot\bm{\Pi}_k^o\nabla y, \widehat{w}_{1}^o\rangle_{\partial \mathcal{T}_h}|\nonumber\\
	&=&\varepsilon|(\Delta y, w_1)_{\mathcal{T}_h}-(\nabla\cdot\bm{\Pi}_k^o\nabla y,\mathcal{I}_{h}(w_1,\widehat{w}_1^o))_{\mathcal{T}_h}+\langle \bm{n}\cdot\bm{\Pi}_k^o\nabla y, \mathcal{I}_{h}(w_1,\widehat{w}_1^o)\rangle_{\partial \mathcal{T}_h}|\nonumber\\
	&=&\varepsilon|(\Delta y, w_1)_{\mathcal{T}_h}+(\bm{\Pi}_k^o\nabla y,\nabla\mathcal{I}_{h}(w_1,\widehat{w}_1^o))_{\mathcal{T}_h}|.\nonumber
	\end{eqnarray}
	Therefore, by the triangle inequality, integration by parts, \eqref{I*3} and \eqref{I*4} we have 
	\begin{align*}
	|R_3|&\le \varepsilon|(\Delta y,{\mathcal{I}}_h(w_1,\widehat{w}_1^o))_{\mathcal{T}_h}+(\bm{\Pi}^o_k\nabla y,{\mathcal{I}}_h(w_1,\widehat{w}_1^o) )_{\mathcal{T}_h}|
	+\varepsilon|(\Delta y, w_1-{\mathcal{I}}_h(w_1,\widehat{w}_1^o))_{\mathcal{T}_h}| \nonumber\\
	&= \varepsilon|(\nabla y-\bm{\Pi}^o_k\nabla y,\nabla{\mathcal{I}}_h(w_1,\widehat{w}_1^o))_{\mathcal{T}_h}|
	+\varepsilon|(\Delta y, w_1-\mathcal{I}_{h}(w_1,\widehat{w}_1^o))_{\mathcal{T}_h}| \nonumber\\
	&\le C \varepsilon^{1/2}( h^{s_y-1}\|y\|_{s_y}+h\|\Delta y\|_{\mathcal{T}_h}  )\norm{(w_1,\widehat{w}_1^o)}_W\nonumber\\
	&\le C (h^{s_y-1/2} \|y\|_{s_y}+\varepsilon^{1/2}h\|\Delta y\|_{\mathcal{T}_h}  )\norm{(w_1,\widehat{w}_1^o)}_W.
	\end{align*}
	For the terms $R_4$ and $R_5$, use the Cauchy-Schwarz inequality and \eqref{condition-beta} to get
	\begin{align*}
	|R_4|&= | \langle \bm{\beta}\cdot\bm{n} (y-\Pi_k^{\partial} y),w_{1}-\widehat{w}_1^{o} \rangle_{\partial \mathcal{T}_h}|\nonumber\\
	&\le C h^{s_y-1/2} \|\bm{\beta}\|^{1/2}_{0,\infty}\|y\|_{s_y}\left(\sum_{T\in\mathcal{T}_h} | \bm{\beta}\cdot\bm{n}|^{1/2}(w_{1}-\widehat{w}_1^{o})\|^2_{\partial T}\right)^{1/2}\nonumber\\
	&\le C h^{s_y-1/2}\|y\|_{s_y}\norm{(w_1,\widehat{w}_1^o)}_W,\\
	|R_5|&= | \langle \tau_1 (\Pi^{o}_k-\Pi_k^{\partial})y,w_{1}-\widehat{w}_1^{o} \rangle_{\partial \mathcal{T}_h}|\nonumber\\
	&\le Ch^{s_y-1}(\beta_0^{1/2}h^{1/2}+\varepsilon^{1/2})\|y\|_{s_y}\norm{(w_1,\widehat{w}_1^o)}_W\nonumber\\
	&\le Ch^{s_y-1/2}\|y\|_{s_y}\norm{(w_1,\widehat{w}_1^o)}_W.
	\end{align*}
	From all the estimates above we get our final result.
\end{proof}
\begin{lemma} \label{theorem4.4}
	Let $\left(\bm q, y,\bm p, z\right)$ and
	$(\bm q_h(u),y_h(u),\widehat{y}^o_h(u), $ $\bm p_h(u),z_h(u),\widehat{z}^o_h(u))\in [\bm V_h\times W_h\times M^{o}_h]^2$ be
	the solutions of \eqref{mixed} and \eqref{Co}, respectively.  If assumptions  \textbf{(A1)} and \textbf{(A2)} hold,  then there exists $h_0>0$, independent of $\varepsilon$,  such that for all $h\le h_0$ we have the error estimates
	\begin{align*}
	\|y-y_h(u)\|_{\mathcal{T}_h}&\le C
	\left(
	h^{s_y-1/2}\|y\|_{s_y}+\delta(s_y)\varepsilon^{1/2}h\|\Delta y\|_{\mathcal{T}_h}
	\right),\\
	\|z-z_h(u)\|_{\mathcal{T}_h}&\le C
	\left(
	h^{s_y-1/2}\|y\|_{s_y}+h^{s_z-1/2}\|z\|_{s_z}+\delta(s_y)\varepsilon^{1/2}h\|\Delta y\|_{\mathcal{T}_h}
	\right) ,\\
	\|\bm p-\bm p_h(u)\|_{\mathcal{T}_h}&\le C
	\varepsilon^{1/2}\left(
	h^{s_y-1/2}\|y\|_{s_y}+h^{s_z-1/2}\|z\|_{s_z}+\delta(s_y)\varepsilon^{1/2}h\|\Delta y\|_{\mathcal{T}_h}
	\right) .
	\end{align*}
\end{lemma}
\begin{proof} 
	By \Cref{LBB}, \eqref{error001} and \eqref{es1} we get
	\begin{align}
	\hspace{1em}&\hspace{-1em}\|(\bm{\Pi}^o_k\bm{q}-\bm{q}_h(u),\Pi^{o}_ky-y_h(u),\widetilde{\Pi}^{\partial}_ky-\widehat{y}_h^o(u))\|\nonumber\\
	&\le C \sup_{{(\bm r_1,w_1,\widehat w_1^o)\in \bm{V}_h\times W_h\times M_h^o}}\frac{\mathcal B_1(\bm{\Pi}^o_k\bm{q}-\bm{q}_h(u),\Pi^{o}_ky-y_h(u),\widetilde{\Pi}^{\partial}_ky-\widehat{y}_h^o(u);\bm r_1,w_1,\widehat w_1^o)}{\norm{(\bm r_1,w_1,\widehat w_1^o)}}\nonumber\\
	&\le C
	\sup_{{(\bm r_1,w_1,\widehat w_1^o)\in \bm{V}_h\times W_h\times M_h^o}}\frac{E_1(\bm{q},y;v_h,\widehat{v}_h^o)
	}{\norm{(\bm r_1,w_1,\widehat w_1^o)}}\nonumber\\
	&\le C \left(h^{s_y-1/2}\|y\|_{s_y}+\delta(s_y)\varepsilon^{1/2}h\|\Delta y\|_{\mathcal{T}_h}\right).\nonumber
	\end{align}
	Therefore,
	\begin{align}
	\hspace{1em}&\hspace{-1em}\ |y-y_h(u)\|_{\mathcal{T}_h}\nonumber\\
	& \le \|y-\Pi^{o}_ky\|_{\mathcal{T}_h} +\|\Pi^{o}_ky-y_h(u)\|_{\mathcal{T}_h}\nonumber\\
	&\le
	\|y-\Pi^{o}_ky\|_{\mathcal{T}_h} +C \|(\bm{\Pi}^o_k\bm{q}-\bm{q}_h(u),\Pi^{o}_ky-y_h(u),\widetilde{\Pi}^{\partial}_ky-\widehat{y}_h^o(u))\| \nonumber \\
	&\le C \left(h^{s_y-1/2}\|y\|_{s_y}+\delta(s_y)\varepsilon^{1/2}h\|\Delta y\|_{\mathcal{T}_h}\right).\label{y-p}
	\end{align}
	By  \Cref{LBB}, \eqref{error002}, \eqref{es2} and estimate \eqref{y-p} we get
	\begin{align}
	\hspace{1em}&\hspace{-1em} \|(\bm{\Pi}^o_k\bm{p}-\bm{p}_h(u),\Pi^{o}_kz-z_h(u),\widetilde{\Pi}^{\partial}_kz-\widehat{z}_h^o(u))\|\nonumber\\
	&\le C
	\sup_{(\bm{r}_2,w_2,\widehat w_2^o)\in \bm{V}_h\times W_h\times M_h^o}\frac{\mathcal B_2(\bm{\Pi}^o_k\bm{p}-\bm{p}_h(u),\Pi^{o}_kz-z_h(u),\widetilde{\Pi}^{\partial}_kz-\widehat{z}_h^o(u);\bm{r}_2,w_2,\widehat w_2^o)}{\norm{(\bm{r}_2,w_2,\widehat w_2^o)}}\nonumber\\
	&\le C
	\sup_{(\bm r_2,w_2,\widehat w_2^o)\in \bm{V}_h\times W_h\times M_h^o}\frac{E_2(\bm{p},z;w_2,\widehat{w}_2^o)
		-(y-y_h(u),w_2)}{\norm{(\bm{r}_2,w_2,\widehat w_2^o)}}\nonumber\\
	&\le C \left(
	h^{s_z-1/2}\|z\|_{s_z}+
	h^{s_y-1/2}\|y\|_{s_y}+\delta(s_y)\varepsilon^{1/2}h\|\Delta y\|_{\mathcal{T}_h}\right).\label{par_z}
	\end{align}
	Therefore, the triangle inequality gives
	\begin{align*}
	&\|z-z_h(u)\|_{\mathcal{T}_h}\nonumber\\
	&\qquad\le \|z-\Pi^{o}_kz\|_{\mathcal{T}_h} +\|\Pi^{o}_kz-z_h(u)\|_{\mathcal{T}_h}\nonumber\\
	&\qquad\le
	\|z-\Pi^{o}_kz\|_{\mathcal{T}_h} +(\beta_0+\sigma_0)^{-1} \|(\bm{\Pi}^o_k\bm{p}-\bm{p}_h(u),\Pi^{o}_kz-z_h(u),\widetilde{\Pi}^{\partial}_kz-\widehat{z}_h^o(u))\| \nonumber\\
	&\qquad\le  C \left(
	h^{s_z-1/2}\|z\|_{s_z}+
	h^{s_y-1/2}\|y\|_{s_y}+\delta(s_y)\varepsilon^{1/2}h\|\Delta y\|_{\mathcal{T}_h}\right).
	\end{align*}
	Next,  we use the triangle inequality, $\bm p=-\varepsilon\nabla z$, and $\varepsilon<\min_{T\in \mathcal T_h}\{h_T\}$  to get
	\begin{align*}
	\|\bm p-\bm p_h(u)\|_{\mathcal{T}_h}&\le \|\bm p-\bm{\Pi}^o_k\bm{p}\|_{\mathcal{T}_h} +\|\bm{\Pi}^o_k\bm{p}-\bm{p}_h(u)\|_{\mathcal{T}_h}\\
	&\le
	\|\bm p-\bm{\Pi}^o_k\bm{p}\|_{\mathcal{T}_h} + \varepsilon^{1/2}\|(\bm{\Pi}^o_k\bm{p}-\bm{p}_h(u),\Pi^{o}_kz-z_h(u),\widetilde{\Pi}^{\partial}_kz-\widehat{z}_h^o(u))\| \\
	&\le C\varepsilon^{1/2}\left(
	h^{s_z-1/2}\|z\|_{s_z}+
	h^{s_y-1/2}\|y\|_{s_y}+\delta(s_y)\varepsilon^{1/2}h\|\Delta y\|_{\mathcal{T}_h}\right).
	\end{align*}
\end{proof}

\subsubsection{Step 2: errors between the  auxiliary problem \eqref{Co} and the  discrete problem \eqref{HDG_full_discrete}}
\label{part II}

By \eqref{HDG_full_discrete} and \eqref{Co} we have  the following error equations.
\begin{lemma} 
	Let $\left(\bm q_h,y_h,\widehat{y}^o_h,\bm p_h,z_h,\widehat{z}^o_h, u_h\right)\in [\bm V_h\times W_h\times M^{o}_h]^2\times M_h^{\partial}$ and $(\bm q_h(u), \\ y_h(u), \widehat{y}^o_h(u),\bm p_h(u),z_h(u),\widehat{z}^o_h(u))\in [\bm V_h\times W_h\times M^{o}_h]^2$ be the solutions of \eqref{HDG_full_discrete} and \eqref{Co}, respectively. Then for all  $\left(\bm r_1,w_1,\widehat{w}^o_1,\bm r_2,w_2,\widehat{w}^o_2\right)\in [\bm V_h\times W_h\times M^{o}_h]^2$, we have 
	\begin{subequations}\label{Error_2}
		\begin{align}
		\mathcal B_1(\bm q_h-\bm q_h(u),y_h-y_h(u),\widehat y^o_h-\widehat y^o_h(u);\bm r_1,w_1,\widehat w_1^o) &=\nonumber\\
		-\langle u_h&-\Pi^{\partial}_ku,\tau_2 w_1+\bm r_1\cdot\bm{n} \rangle_{\mathcal{E}_h^\partial},\label{Error_2_01}\\
		\mathcal B_2(\bm p_h-\bm p_h(u),z_h-z_h(u),\widehat z^o_h-\widehat z^o_h(u);\bm r_2,w_2,\widehat w_2^o)&=-(y_h-y_h(u),w_2)_{\mathcal{T}_h}.\label{Error_2_02}
		\end{align}
	\end{subequations}
\end{lemma}	
\begin{lemma} \label{lemma_u_par}
	Let $\left(\bm q_h,y_h,\widehat{y}^o_h,\bm p_h,z_h,\widehat{z}^o_h\right)\in [\bm V_h\times W_h\times M^{o}_h]^2$ and
	$(\bm q_h(u),y_h(u), \\ \widehat{y}^o_h(u), \bm p_h(u),z_h(u),\widehat{z}^o_h(u))\in [\bm V_h\times W_h\times M^{o}_h]^2$ be
	the solutions of \eqref{HDG_full_discrete} and \eqref{Co}, respectively.  Then we have 
	\begin{align*}
	&\|y_h-y_h(u)\|^2_{\mathcal{T}_h}+\gamma\|u_h- \Pi^{\partial}_ku\|^2_{\mathcal E_h^\partial}\nonumber\\
	&\qquad\qquad=\langle u_h-\Pi^{\partial}_ku,\Pi^{\partial}_k(\bm{p}\cdot\bm{n})-\bm p_h(u)\cdot\bm{n}-\tau_2(z_h(u)-\widehat{z}_h^o(u)) \rangle_{\mathcal{E}_h^\partial}.
	\end{align*}
\end{lemma}
\begin{proof}
	Take $(\bm r_1,w_1,\widehat w_1^o)=(\bm p_h-\bm p_h(u),z_h-z_h(u),\widehat z^o_h-\widehat z^o_h(u))$ and
	$(\bm r_2,w_2,\widehat w_2^o)$ $=(\bm q_h-\bm q_h(u),y_h-y_h(u),\widehat y^o_h-\widehat y^o_h(u))$ in
	\eqref{Error_2_01} and \eqref{Error_2_02}, respectively, and use \Cref{sys} to get
	\begin{align*}
	&-\|y_h-y_h(u)\|^2_{\mathcal{T}_h}\nonumber\\
	&\qquad=\mathcal B_2(\bm p_h-\bm p_h(u),z_h-z_h(u),\widehat z^o_h-\widehat z^o_h(u);\bm q_h-\bm q_h(u),y_h-y_h(u),\widehat y^o_h-\widehat y^o_h(u))\nonumber\\
	&\qquad=\mathcal B_1(\bm q_h-\bm q_h(u),y_h-y_h(u),\widehat y^o_h-\widehat y^o_h(u);\bm p_h-\bm p_h(u),z_h-z_h(u),\widehat z^o_h-\widehat z^o_h(u))\nonumber\\
	&\qquad=-\langle u_h-\Pi^{\partial}_ku,( \bm p_h-\bm p_h(u))\cdot\bm{n}+\tau_2(z_h-z_h(u) ) \rangle_{\mathcal{E}_h^\partial}.
	\end{align*}
	Therefore, \eqref{6e}, \eqref{mixed} and $\widehat{z}^o_h=\widehat{z}^o_h(u)=0$ on $\mathcal E_h^\partial$ give
	\begin{align*}
	&\|y_h-y_h(u)\|^2_{\mathcal{T}_h}+\gamma\|u_h-\Pi^{\partial}_ku\|^2_{\mathcal E_h^\partial}\nonumber\\
	&\qquad=\langle u_h-\Pi^{\partial}_ku, \bm p_h\cdot\bm{n}+\tau_2z_h+\gamma u_h \rangle_{\mathcal{E}_h^\partial} \nonumber\\
	&\qquad\quad-\langle u_h-\Pi^{\partial}_ku,\bm p_h(u)\cdot\bm{n}+\tau_2z_h(u) +\gamma \Pi^{\partial}_ku \rangle_{\mathcal{E}_h^\partial} \nonumber\\
	&\qquad=\langle u_h-\Pi^{\partial}_ku,\Pi^{\partial}_k(\bm{p}\cdot\bm{n})-\bm p_h(u)\cdot\bm{n}-\tau_2(z_h(u)-\widehat{z}_h^o(u)) \rangle_{\mathcal{E}_h^\partial}. \nonumber
	\end{align*}
\end{proof}

\begin{lemma}\label{error_ee}
	Let $\left(\bm q_h,y_h,\widehat{y}^o_h,\bm p_h,z_h,\widehat{z}^o_h\right)\in [\bm V_h\times W_h\times M^{o}_h]^2$ and
	$(\bm q_h(u),y_h(u),$ $\widehat{y}^o_h(u),\bm p_h(u),z_h(u),\widehat{z}^o_h(u))\in [\bm V_h\times W_h\times M^{o}_h]^2$ be the solutions of \eqref{HDG_full_discrete} and \eqref{Co}, respectively.  If assumptions  \textbf{(A1)}-\textbf{(A3)} hold,  then there exists $h_0$, independent of $\varepsilon$ such that for all $h\le h_0$,  we have
	the estimates
	\begin{align*}
	\|u-u_h\|_{\mathcal E_h^\partial}&\le C
	\left(
	h^{s_{y}-1/2}\|y\|_{s_y}+ h^{s_{z}-1/2}\|z\|_{s_z}+\delta(s_y)\varepsilon^{1/2}h\|\Delta y\|_{\mathcal{T}_h}
	\right)
	,\\
	\|y_h-y_h(u)\|_{\mathcal{T}_h}&\le C
	\left(
	h^{s_{y}-1/2}\|y\|_{s_y}+ h^{s_{z}-1/2}\|z\|_{s_z}+\delta(s_y)\varepsilon^{1/2}h\|\Delta y\|_{\mathcal{T}_h}
	\right)
	,\\
	\|z_h-z_h(u)\|_{\mathcal{T}_h}&\le C\left(
	h^{s_{y}-1/2}\|y\|_{s_y}+ h^{s_{z}-1/2}\|z\|_{s_z}+\delta(s_y)\varepsilon^{1/2}h\|\Delta y\|_{\mathcal{T}_h}
	\right).
	\end{align*}
\end{lemma}
\begin{proof} By \Cref{lemma_u_par}, the Cauchy-Schwarz inequality, and Young's inequality one gets
	\begin{align*}
	\hspace{1em}&\hspace{-1em}\|y_h-y_h(u)\|_{\mathcal{T}_h}+\gamma^{1/2}\|u_h- \Pi^{\partial}_ku\|_{\mathcal E_h^\partial}\\
	&\qquad\le C\|\Pi^{\partial}_k(\bm{p}\cdot\bm{n})-\bm p_h(u)\cdot\bm{n}\|_{\mathcal{E}_h^{\partial}}+C\|\tau_2(z_h(u)-\widehat{z}_h^o(u)\|_{\mathcal{E}_h^{\partial}}.
	\end{align*}
	By the  triangle inequality,  $\bm p=-\varepsilon \nabla z$, an  inverse inequality and the estimates in \Cref{error_ee} we have
	\begin{align*}
	\hspace{1em}&\hspace{-1em}\|\bm{\Pi}^{\partial}_k(\bm{p}\cdot\bm{n})-\bm p_h(u)\cdot\bm{n}\|_{\mathcal{E}_h^{\partial}}\nonumber\\
	&\le
	\|\bm{\Pi}^{\partial}_k(\bm{p}\cdot\bm{n})-\bm{\Pi}^{o}_k(\bm{p}\cdot\bm{n})\|_{\mathcal{E}_h^{\partial}}+\|\bm{\Pi}^{o}_k(\bm{p}\cdot\bm{n})-\bm p_h(u)\cdot\bm{n}\|_{\mathcal{E}_h^{\partial}} \nonumber\\
	&\qquad\le C \varepsilon h^{s_{z}-3/2}\|z\|_{s_z}+\sum_{T\in\mathcal{T}_h,\partial T\bigcap \mathcal{E}_h^{\partial}\neq\emptyset}
	h_T^{-1/2}\|\bm \Pi^{o}_k\bm{p}-\bm p_h(u)\|_{T} \nonumber\\
	&\qquad\le C\varepsilon h^{s_{z}-3/2}\|z\|_{s_z}
	+\sum_{T\in\mathcal{T}_h,\partial T\bigcap \mathcal{E}_h^{\partial}\neq\emptyset}
	h_T^{-1/2}
	\left(
	\|\bm \Pi^{o}_k\bm{p}-\bm p\|_{T}
	+\|\bm{p}-\bm p_h(u)\|_{T}
	\right) \nonumber\\
	&\qquad\le  Ch^{s_{z}-1/2}\|z\|_{s_z}
	+\varepsilon^{-1/2}\sum_{T\in\mathcal{T}_h,\partial T\bigcap \mathcal{E}_h^{\partial}\neq\emptyset}
	\|\bm{p}-\bm p_h(u)\|_{T} \nonumber\\
	&\qquad\le  Ch^{s_{z}-1/2}\|z\|_{s_z}
	+\varepsilon^{-1/2}
	\|\bm{p}-\bm p_h(u)\|_{\mathcal{T}_h} \nonumber\\
	&\qquad\le C
	\left(
	h^{s_y-1/2}\|y\|_{s_y}+h^{s_z-1/2}\|z\|_{s_z}+\delta(s_y)\varepsilon^{1/2}h\|\Delta y\|_{\mathcal{T}_h}
	\right) .
	\end{align*}
	By the  triangle inequality, $z=0$ on $\Gamma$,  and the estimate \eqref{par_z} we have 
	\begin{align*}	
	\hspace{1em}&\hspace{-1em}\|\tau_2(z_h(u)-\widehat{z}_h^o(u)\|_{\mathcal{E}_h^{\partial}}\nonumber\\
	&\le C \|(z_h(u)-\widehat{z}_h^o(u)-\Pi^o_kz+\widetilde{\Pi}_k^{\partial}z)\|_{\mathcal{E}_h^{\partial}}
	+ C\|(\Pi^o_kz-\widetilde{\Pi}_k^{\partial}z)\|_{\mathcal{E}_h^{\partial}} \nonumber\\
	&\le C\|(\bm{\Pi}^o_k\bm{p}-\bm{p}_h(u),\Pi^{o}_kz-z_h(u),\widetilde{\Pi}^{\partial}_kz-\widehat{z}_h^o(u))\|+Ch^{s_z-1/2}\|z\|_{s_z}\nonumber\\
	&\le  C\left( h^{s_y-1/2}\|y\|_{s_y}+h^{s_z-1/2}\|z\|_{s_z}+\delta(s_y)\varepsilon^{1/2}h\|\Delta y\|_{\mathcal{T}_h}\right).
	\end{align*}	
	This implies
	\begin{align*}
	\hspace{1em}&\hspace{-1em}\|y_h-y_h(u)\|_{\mathcal{T}_h}
	+\|u_h- \Pi^{\partial}_ku\|_{\mathcal E_h^\partial}\nonumber\\
	&\le C\left( h^{s_{y}-1/2}\|y\|_{s_y}+
	h^{s_{z}-1/2}\|z\|_{s_z}+\delta(s_y)\varepsilon^{1/2}h\|\Delta y\|_{\mathcal{T}_h}\right).
	\end{align*}
	By the triangle inequality and  the fact $y=u$ on $\mathcal E_h^\partial$, we get 
	\begin{align*}
	\|u-u_h\|_{\mathcal E_h^\partial}&\le
	\|y-\Pi_k^{\partial }y\|_{\mathcal E_h^\partial}+\|\Pi_k^{\partial }u-u_h\|_{\mathcal E_h^\partial}\nonumber\\
	&\le C
	\left(
	h^{s_{y}-1/2}\|y\|_{s_y}+ h^{s_{z}-1/2}\|z\|_{s_z}+\delta(s_y)\varepsilon^{1/2}h\|\Delta y\|_{\mathcal{T}_h}
	\right).
	\end{align*}
	By \Cref{LBB} and \eqref{Error_2_02}, one has
	\begin{align}
	\hspace{1em}&\hspace{-1em}\|(\bm{p}_h(u)-\bm{p}_h,z_h(u)-z_h,\widehat{z}_h^o(u)-\widehat{z}_h^o)\| \nonumber\\
	&\le C
	\sup_{(\bm{r}_2,w_2,\widehat w_2^o)\in \bm{V}_h\times W_h\times M_h^o}\frac{\mathcal B_2(\bm{p}_h(u)-\bm{p}_h,z_h(u)-z_h,\widehat{z}_h^o(u)-\widehat{z}_h^o;\bm{r}_2,w_2,\widehat w_2^o)}{\norm{(\bm{r}_2,w_2,\widehat w_2^o)} }\nonumber\\
	&\le C
	\sup_{(\bm{r}_2,w_2,\widehat w_2^o)\in \bm{V}_h\times W_h\times M_h^o}\frac{(y_h-y_h(u),w_2)_{\mathcal{T}_h}
	}{\norm{(\bm{r}_2,w_2,\widehat w_2^o)} }\nonumber\\
	&\le C \| y_h-y_h(u)\|_{\mathcal{T}_h} \nonumber\\
	&\le  C
	\left(h^{s_y-1/2}\|y\|_{s_y}+h^{s_z-1/2}\|z\|_{s_z}+\delta(s_y)\varepsilon^{1/2}h\|\Delta y\|_{\mathcal{T}_h}\right).\nonumber
	\end{align}
	Therefore,
	\begin{align*}
	\|z_h(u)-z_h\|_{\mathcal{T}_h}
	\le C\varepsilon^{1/2}\left(
	h^{s_y-1/2}\|y\|_{s_y}+h^{s_z-1/2}\|z\|_{s_z}+\delta(s_y)\varepsilon^{1/2}h\|\Delta y\|_{\mathcal{T}_h}
	\right) .
	\end{align*}
\end{proof}

\section{Numerical Experiments}
\label{sec:numerics}
In this  section, we report numerical experiments to illustrate our theoretical results.  For all experiments, we take $\Omega = [0,1]\times [0,1] \subset \mathbb{R}^2$,  $\gamma = 1$,  and the stabilization functions are chosen as in \eqref{def_tau1}-\eqref{def_tau2}.

\subsection{Smooth test}
In our first test, the state, dual state, and convection coefficient are chosen as
\begin{gather*}
y=-\varepsilon\pi(\sin(\pi x_1)+\sin(\pi x_2)), \  z =  \sin(\pi x_1)\sin(\pi x_2), \\
\bm \beta =-[x_1^2\sin(x_2),\cos(x_1) e^{x_2}],
\end{gather*}
and the source term $f$ and the desired state $y_d$ are generated using the optimality system \eqref{eq_adeq} with the above data. We show the numerical results for $ k = 1 $ and $\varepsilon=10^{-7}$ in \Cref{table_1}.

\begin{table}
	\begin{center}
		\begin{tabular}{cccccc}
			\hline
			$h/\sqrt{2}$ &$2^{-1}$& $2^{-2}$&$2^{-3}$&$2^{-4}$ & $2^{-5}$ \\
			\hline
			$\norm{{y}-{y}_h}_{0,\Omega}$&6.0299E-02   &1.3188E-02   &2.1788E-03   &4.8975E-04   &1.1863E-04 \\
			\hline
			order&-& 2.1929   &2.5976   &2.1534   &2.0456 \\
			\hline
			$\norm{{z}-{z}_h}_{0,\Omega}$&   1.0572E-01  & 2.6724E-02   &6.2451E-03   &1.5091E-03   &3.7092E-04\\
			\hline
			order&-&  1.9841   &2.0973   &2.0491   &2.0245\\
			\hline
			$\norm{{u}-{u}_h}_{0,\Gamma}$&2.5537E-01  & 5.6029E-02   &1.2108E-02   &2.8176E-03   &6.7424E-04 \\
			\hline
			order&-&  2.1883   &2.2102   &2.1034   &2.0631\\
			\hline
		\end{tabular}
	\end{center}
	\caption{Smooth test with $k=1$ and $\varepsilon = 10^{-7}$: Errors for the control $u$, state $y$ and the  adjoint state $z$.}\label{table_1}
\end{table}

\subsection{Non-smooth test}
Next, we choose the data as 
\begin{align*}
y_d = x(1-x)y(1-y), \ \ f=0, \  \mbox{and} \  \ \   \bm \beta =-[x_1^2\sin(x_2),\cos(x_1) e^{x_2}].
\end{align*}
We tested $5$ cases with different values for $\varepsilon$ and we do not have exact solutions for these problems;  we solved the problems numerically for a triangulation with approximately  $1.5$ million elements and compared these reference solutions against other solutions computed on meshes with larger $h$. The numerical results are shown in \Cref{table_2}; the computed convergence rates are erratic and do not follow a clear pattern.  The same phenomenon has been observed in another work on a convection dominated Dirichlet boundary control problem \cite{PeterBenner}. Also, we plot the state, dual state and boundary control in \Cref{fig:u1,fig:u2,fig:u3,fig:u4,fig:u5}. Furthermore, many works on convection dominated PDEs observe well-behaved convergence if they remove a small portion of the domain containing  the layer; see \cite[Section 6]{MR2595051} for a convection dominated distributed optimal control problem and \cite[Table 4 in Section 5.4]{MR3342199} for a convection dominated PDEs. We did not  to compute the rates by removing the layer since the layer is always on the boundary; see   
\Cref{fig:u1,fig:u2,fig:u3,fig:u4,fig:u5}.

\begin{table}
	\begin{tabular}{c|c|c|c|c|c|c|c}
		\Xhline{0.1pt}
		
		\multirow{2}{*}{$\varepsilon$}
		&\multirow{2}{*}{$\frac{h}{\sqrt{2}}$}	
		&\multicolumn{2}{c|}{$\| y - y_h\|_{\mathcal T_h}$}	
		&\multicolumn{2}{c|}{$\| z - z_h\|_{\mathcal T_h}$}	
		&\multicolumn{2}{c}{$\| u- u_h\|_{\varepsilon_h^\partial}$}	\\
		\cline{3-8}
		& &Error &Rate
		&Error &Rate
		&Error &Rate
		\\
		\cline{1-8}
		\multirow{5}{*}{ $1/10$}
		&$2^{-1}$	&4.8530E-04 	    &	       &1.0783E-03	    &	         &2.0808E-03&\\
		&$2^{-2}$	&1.6439E-04     &1.56	&5.3340E-04     &1.01	   &7.2793E-04   	&1.51\\
		&$2^{-3}$	&5.5062E-05	&1.58	   & 2.0652E-04 	 &1.37   &2.6827E-04&1.44\\
		&$2^{-4}$	&1.5885E-05 	&1.79 	&  6.4955E-05	 &1.67	   &9.1834E-05 	&1.54\\
		&$2^{-5}$	&4.1123E-06     &1.95	&   1.7687E-05     &1.88 	   &2.7451E-05  	&1.74\\

		\cline{1-8}
		\multirow{5}{*}{$1/50$}
		&$2^{-1}$	& 1.3406E-03 & 	       & 2.2114E-03 & 	      & 4.1984E-03        &\\
		&$2^{-2}$	&4.9404E-04&1.44	&1.3080E-03 &0.75	&1.8165E-03	&1.21\\
		&$2^{-3}$	& 2.4238E-04	&1.02	&7.3661E-04&0.83	& 8.9345E-04 &1.02\\
		&$2^{-4}$	& 1.1777E-04	&1.04	& 3.5878E-04	&1.04	& 5.2980E-04 &0.75\\
		&$2^{-5}$	& 4.1387E-05   &1.50	&1.4389E-04   	&1.32	& 2.6837E-04	&0.98\\

		\cline{1-8}
		\multirow{5}{*}{$1/100$}
		&$2^{-1}$	&1.6216E-03  & 	       & 2.5742E-03& 	      & 4.7465E-03       &\\
		&$2^{-2}$	&6.6643E-04&1.28	&1.5857E-03&0.70	&2.1124E-03	&1.17\\
		&$2^{-3}$	& 2.8129E-04	&1.24	&9.2685E-04 &0.77	&1.0435E-03  &1.02\\
		&$2^{-4}$	&  1.7254E-04	&0.71	&5.2512E-04	&0.82	&7.3691E-04&0.50\\
		&$2^{-5}$	& 8.5868E-05   &1.00	&2.5477E-04  &1.04	&  5.0212E-04   	&0.53\\

		\cline{1-8}
		\multirow{5}{*}{$1/1000$}
		&$2^{-1}$	&1.8500E-03 & 	       &  2.9440E-03& 	      &  5.3059E-03     &\\
		&$2^{-2}$	&8.3498E-04 &1.14	&2.0237E-03&0.54	&2.4554E-03	&1.11\\
		&$2^{-3}$	& 3.9689E-04	&1.07	&1.3617E-03&0.57	&1.0865E-03  &1.17\\
		&$2^{-4}$	&2.7140E-04	&0.55	& 9.1122E-04 	&0.58	& 5.2906E-04 &1.03\\
		&$2^{-5}$	& 1.5180E-04&0.84	&5.8115E-04&0.65	&  4.5760E-04	&0.21\\

		\cline{1-8}
		\multirow{5}{*}{$1/10000$}
		&$2^{-1}$	&1.9431E-03 & 	       &  3.0087E-03& 	      & 5.3675E-03   &\\
		&$2^{-2}$	&8.4727E-04&1.20	&2.0778E-03&0.53	&2.5082E-03	&1.10\\
		&$2^{-3}$	& 3.9498E-04 	&1.10	&1.4261E-03	&0.54	&1.1393E-03    &1.14\\
		&$2^{-4}$	& 2.3736E-04	&0.73	& 9.9958E-04 	&0.51	&5.6195E-04&1.01\\
		&$2^{-5}$	& 1.5931E-04   &0.58	&6.9841E-04 	&0.52	& 3.2155E-04	&0.81\\

		\Xhline {0.1 pt}
	\end{tabular}
	\caption{Non-smooth test with different $\varepsilon$ and $k=1$: Errors for the control $u$, state $y$ and the  adjoint state $z$.}\label{table_2}
\end{table}

\begin{figure}
	\centerline{
		\hbox{\includegraphics[width=2.5in]{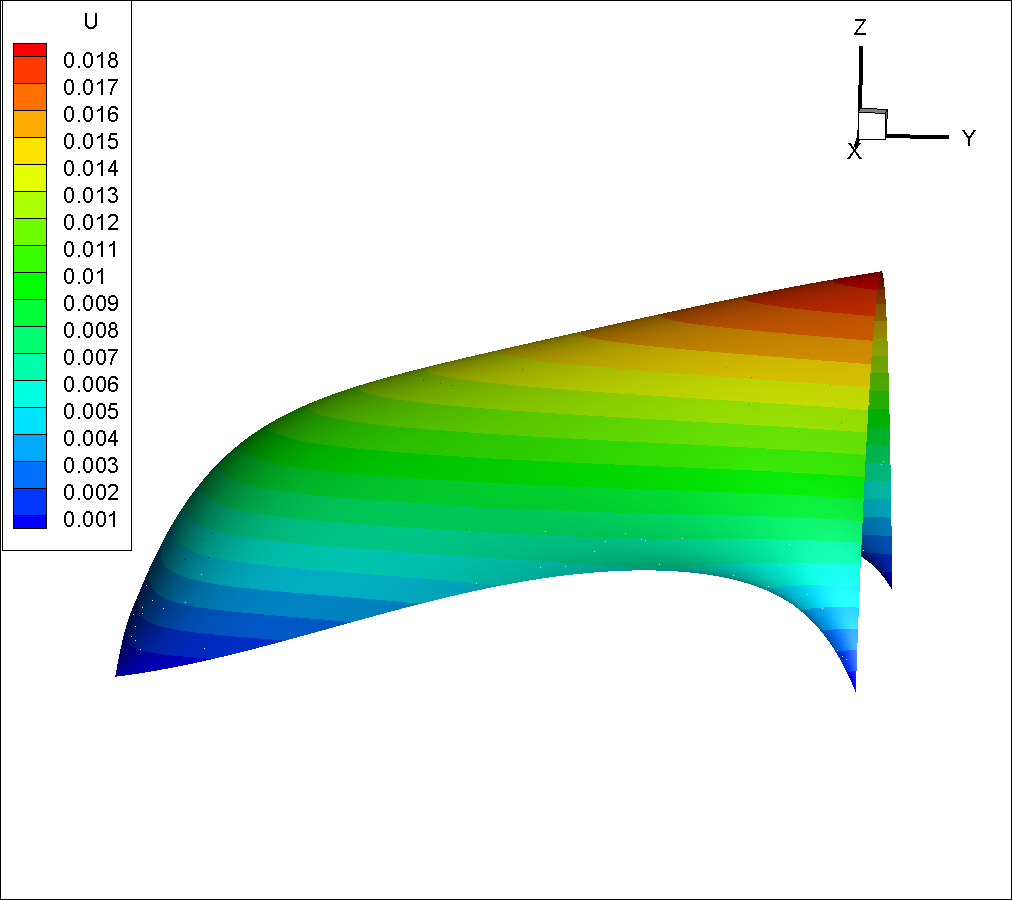}}
		\hbox{\includegraphics[width=2.5in]{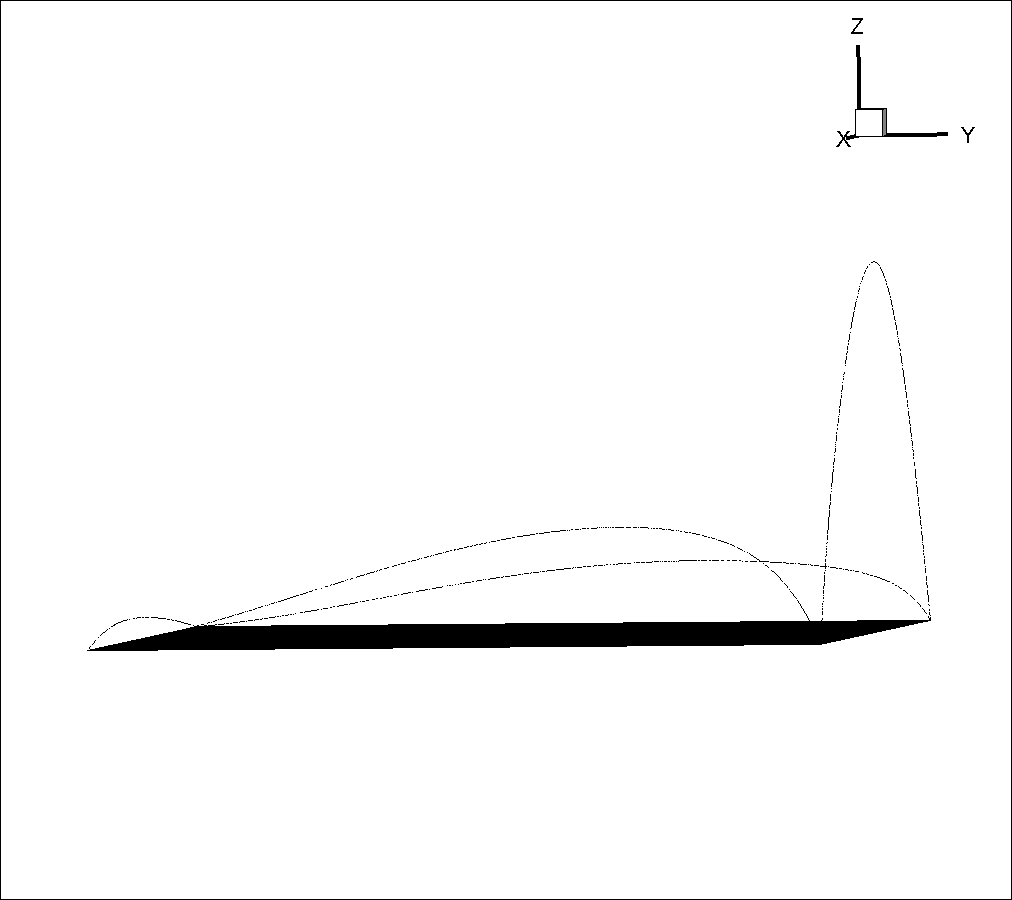}}
	}
	\caption{\label{fig:u1} Left is the state  $y_h$ and right is the control  $u_h$ for $\varepsilon=1/10$.}
\end{figure}
\begin{figure}
	\centerline{
		\hbox{\includegraphics[width=2.5in]{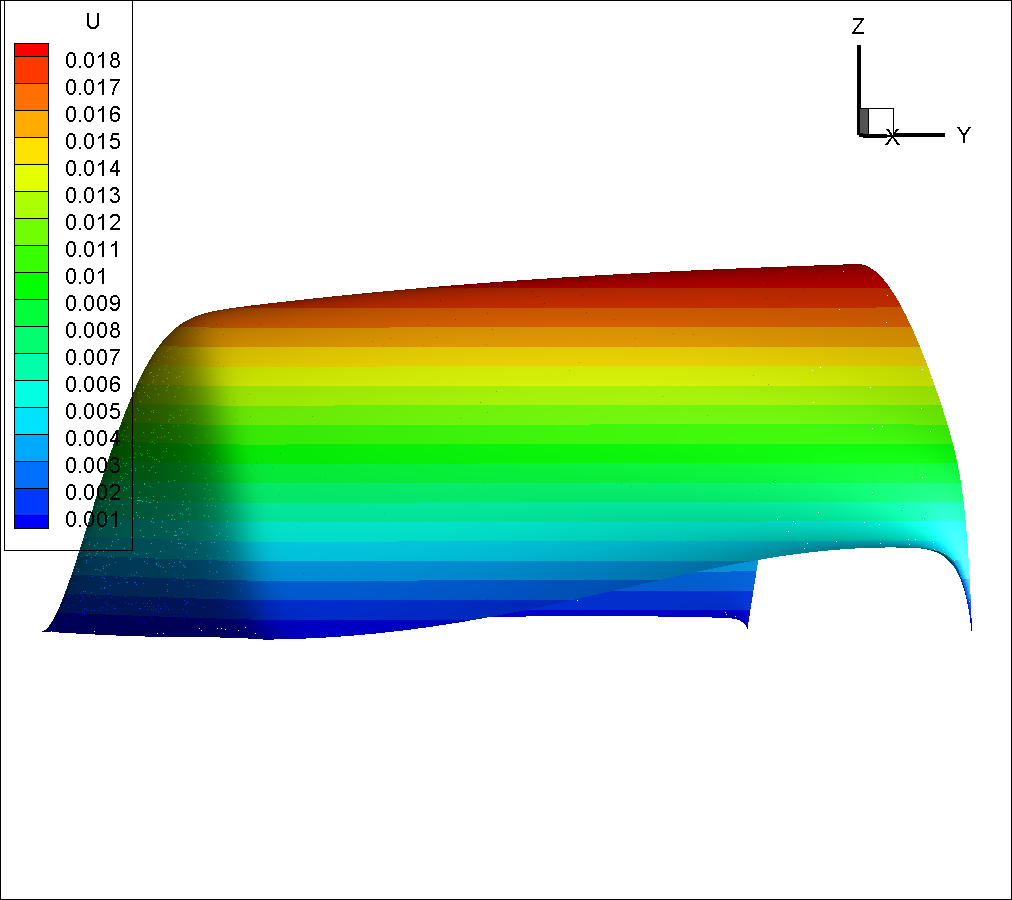}}
		\hbox{\includegraphics[width=2.5in]{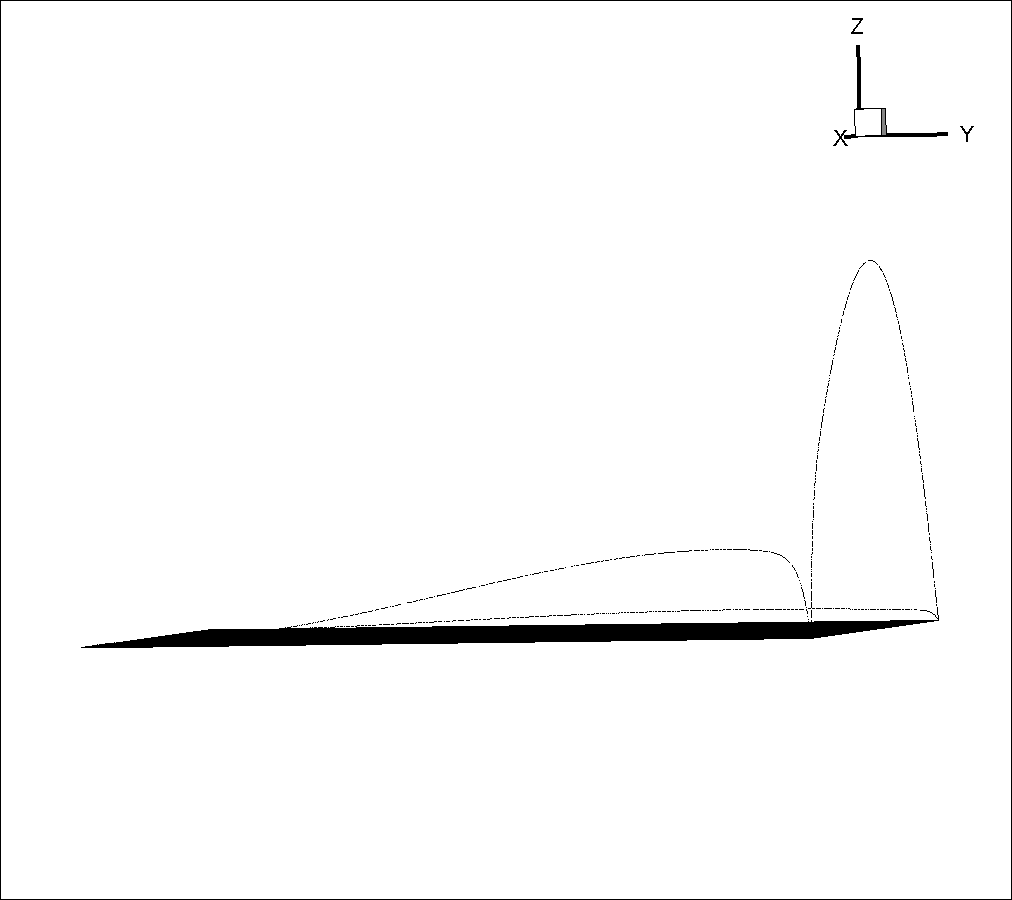}}
	}
	
	\caption{\label{fig:u2} Left is the state  $y_h$ and right is the control  $u_h$ for $\varepsilon=1/50$.}
\end{figure}
\begin{figure}
	\centerline{
		\hbox{\includegraphics[width=2.5in]{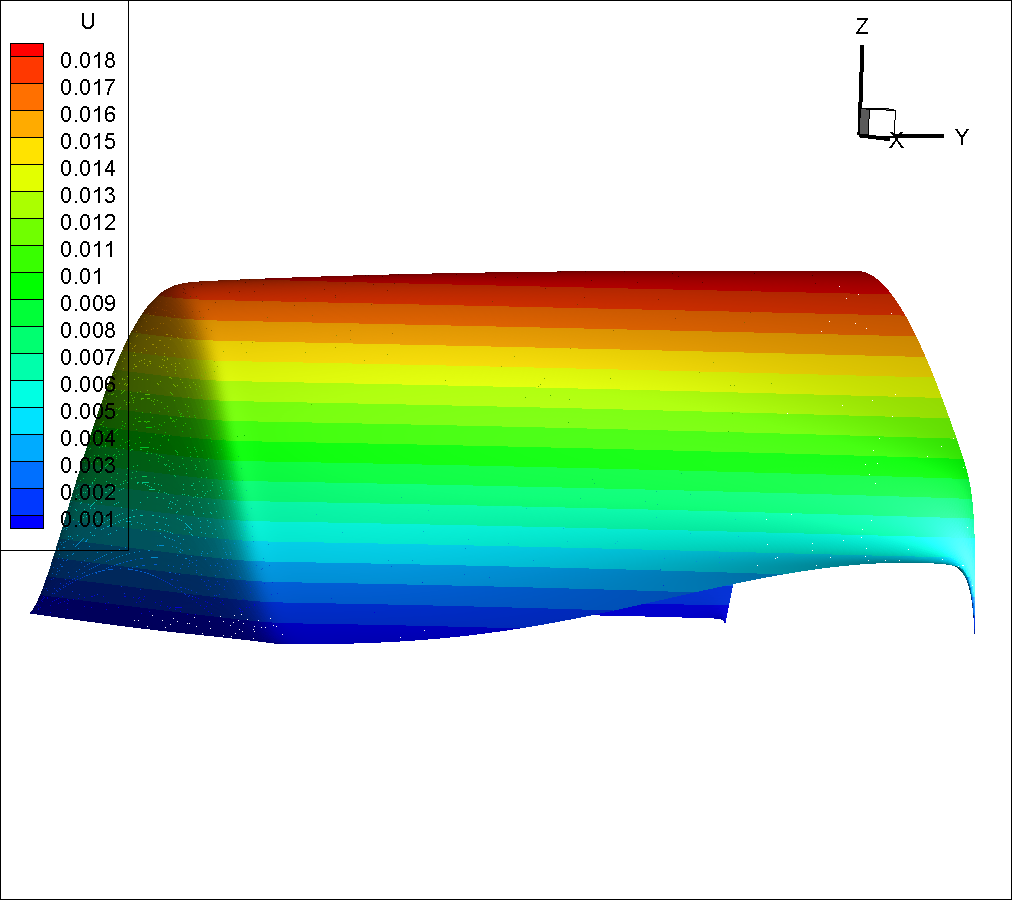}}
		\hbox{\includegraphics[width=2.5in]{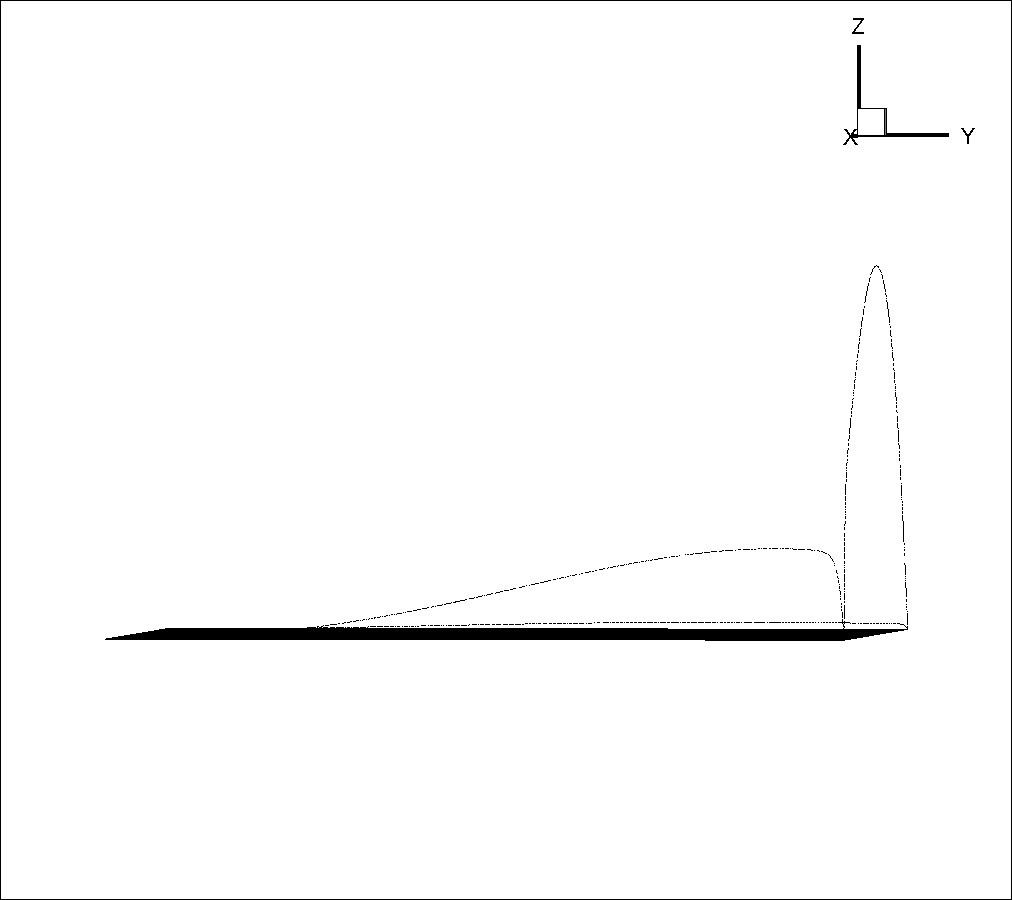}}
	}
	\caption{\label{fig:u3} Left is the state  $y_h$ and right is the control  $u_h$ for $\varepsilon=1/100$.}
\end{figure}

\begin{figure}
	\centerline{
		\hbox{\includegraphics[width=2.5in]{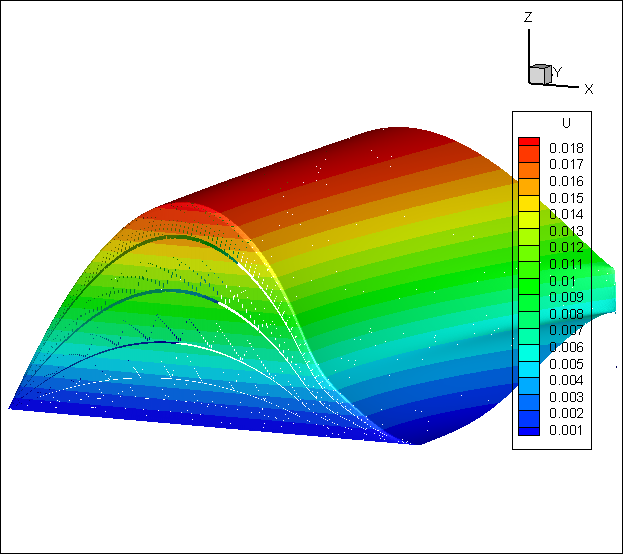}}
		\hbox{\includegraphics[width=2.5in]{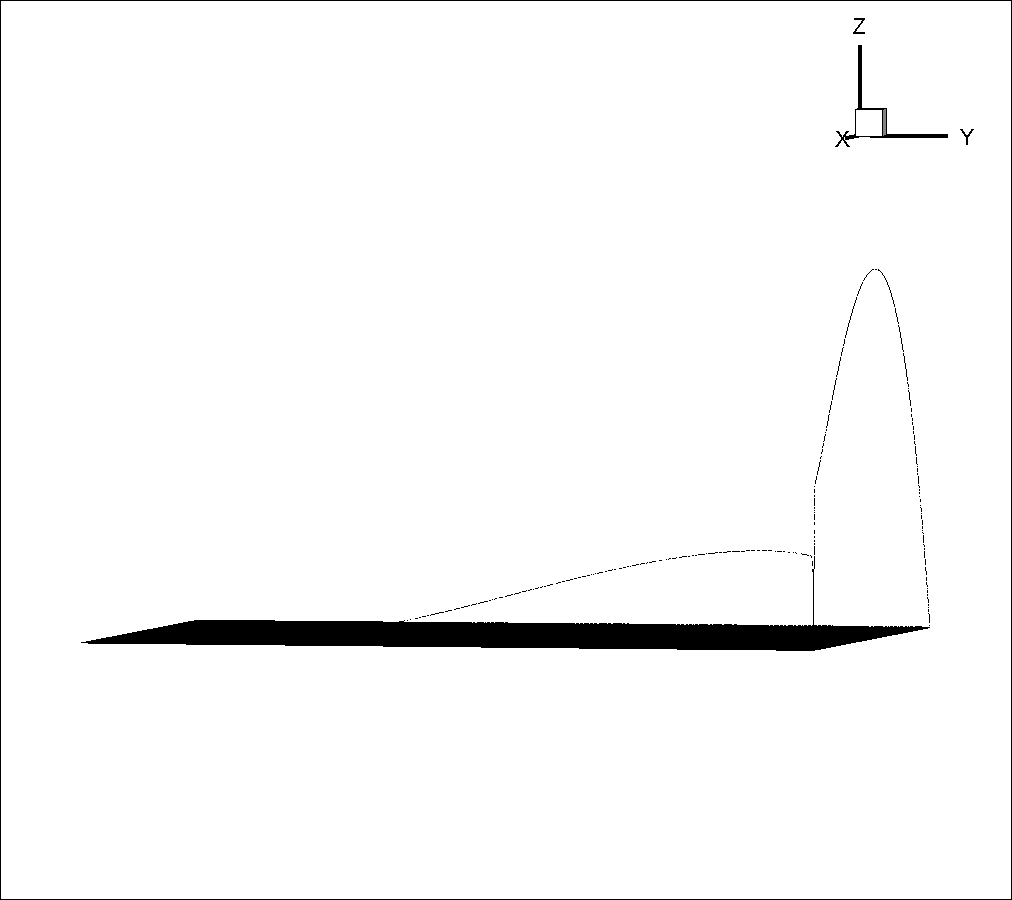}}
	}
	\caption{\label{fig:u4} Left is the state  $y_h$ and right is the control  $u_h$ for $\varepsilon=1/1000$.}
\end{figure}

\begin{figure}
	\centerline{
		\hbox{\includegraphics[width=2.5in]{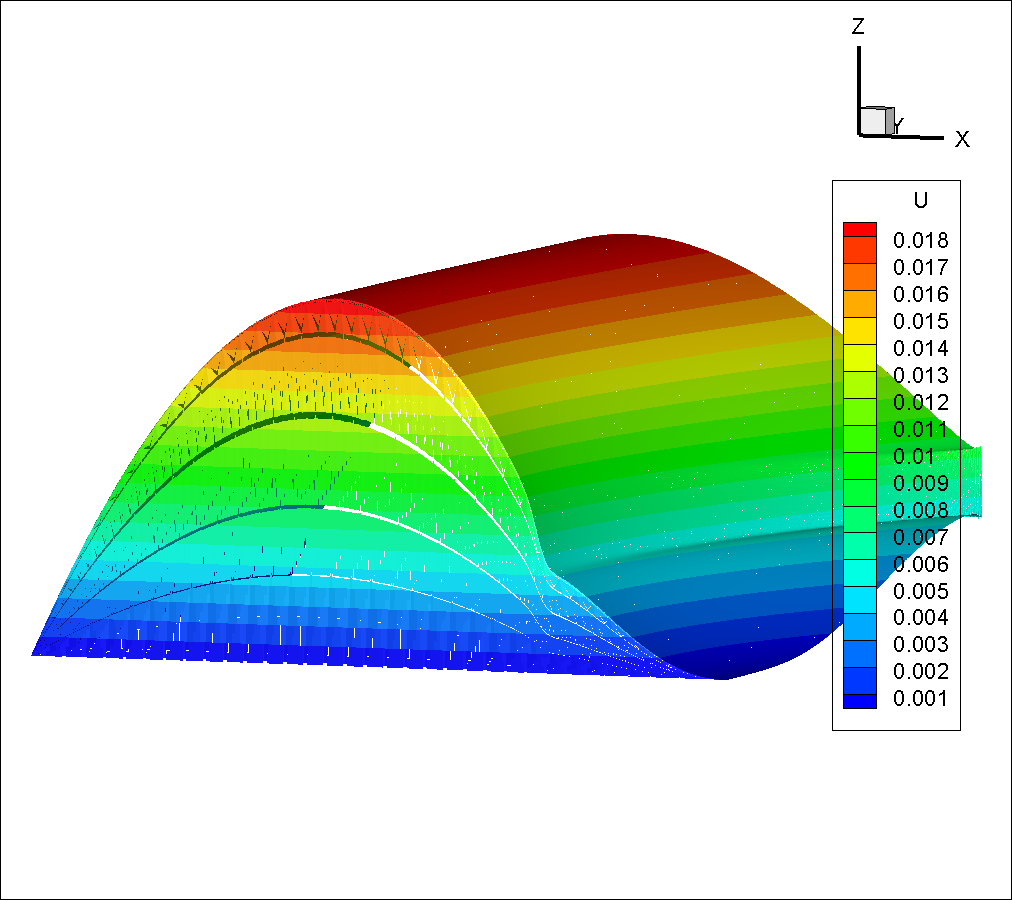}}
		\hbox{\includegraphics[width=2.5in]{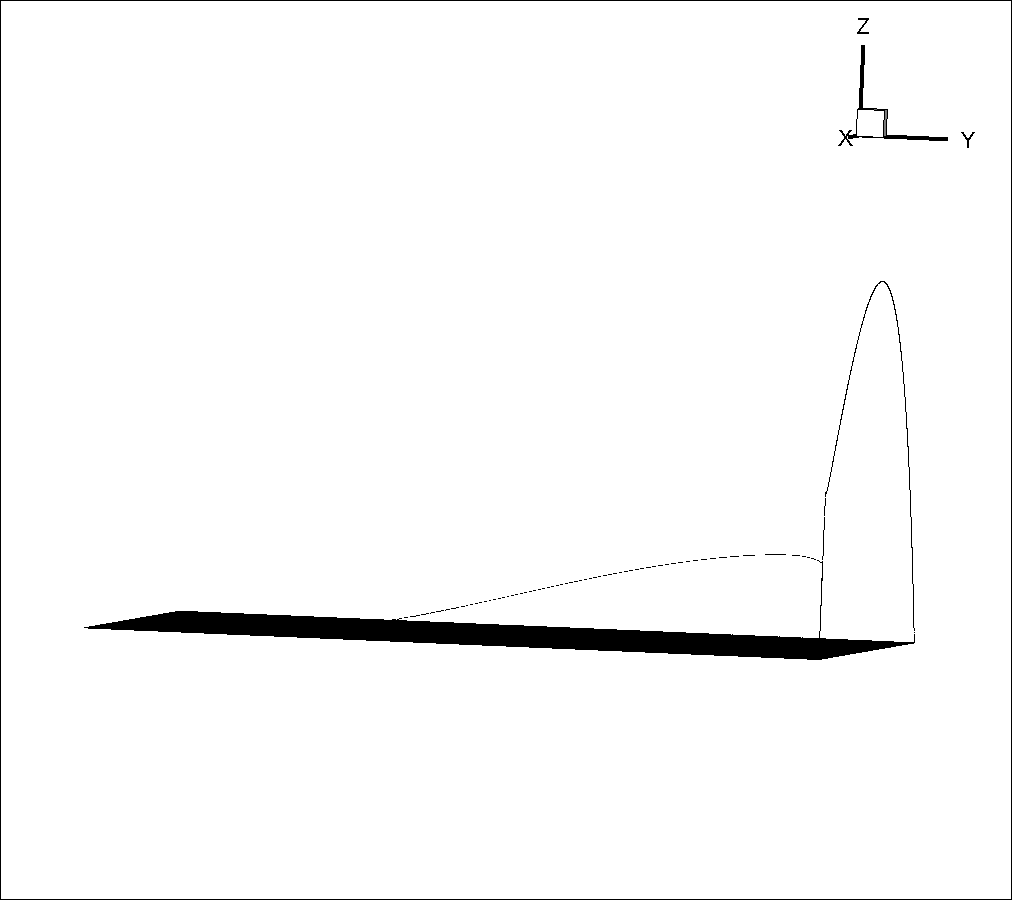}}
	}
	\caption{\label{fig:u5} Left is the state  $y_h$ and right is the control  $u_h$ for $\varepsilon=1/10000$.}
\end{figure}

\section{Conclusion}
In \cite{HuMateosSinglerZhangZhang1,HuMateosSinglerZhangZhang2}, we studied an HDG method for a  \emph{diffuison dominated} convection diffusion  Dirichlet boundary control problem. We obtained optimal convergence rates for the control under a high regularity assumption in \cite{HuMateosSinglerZhangZhang1} and a low regularity assumption in \cite{HuMateosSinglerZhangZhang2}. In this work, we considered a different HDG method with a lower computational cost for a {convection dominated} convection diffusion boundary control problem under high and low regularity conditions and again proved optimal convergence rates for the control. All existing numerical analysis work on Dirichlet boundary control problems have assumed the mesh is quasi-uniform; however, we do not need to have this assumption here.

To the best of our knowledge, this work is the only existing numerical analysis exploration of this convection dominated diffusion Dirichlet control problem.

\bibliographystyle{plain}
\bibliography{yangwen_ref_papers}

\end{document}